\documentclass[12pt,reqno]{amsart}

\usepackage{amssymb}
\usepackage{enumitem}
\usepackage{tikz-cd}

\newtheorem{theorem}{Theorem}[section]
\newtheorem{proposition}[theorem]{Proposition}
\newtheorem{conj}[theorem]{Conjecture}
\newtheorem{bigthm}{Theorem}
 
\newtheorem{bigcor}[bigthm]{Corollary}

\newtheorem{lemma}[theorem]{Lemma}
\newtheorem{corollary}[theorem]{Corollary}

\theoremstyle{definition}
\newtheorem{definition}[theorem]{Definition}
\newtheorem{remark}[theorem]{Remark}

\numberwithin{equation}{section}

\newcommand{\om}{\omega}
\newcommand{\rank}{\rm rank}
\newcommand\cF{{\mathcal F}}
\newcommand\cG{{\mathcal G}}

\newcommand{\R}{\ensuremath{\mathbb{R}}}
\let\ep=\varepsilon

\usepackage{xcolor}
\definecolor{violet}{rgb}{0.0,0.2,0.7}
\definecolor{rouge2}{rgb}{0.8,0.0,0.2}
\usepackage{hyperref}
\hypersetup{
    bookmarks=true,         
    unicode=false,          
    pdftoolbar=true,        
    pdfmenubar=true,        
    pdffitwindow=false,     
    pdfstartview={FitH},    
    pdftitle={},    
    pdfauthor={},     
    colorlinks=true,       
   linkcolor=rouge2,          
    citecolor=violet,        
    filecolor=black,      
    urlcolor=cyan}           

\begin{document}

\baselineskip=15.5pt

\title[Geometry of $K$-trivial Moishezon manifolds]{Geometry of $K$-trivial Moishezon 
manifolds : decomposition theorem and holomorphic geometric structures}

\author[I. Biswas]{Indranil Biswas}

\address{Department of  Mathematics, Shiv Nadar University, NH91, Tehsil
Dadri, Greater Noida, Uttar Pradesh 201314, India}

\email{indranil.biswas@snu.edu.in, indranil29@gmail.com}

\author[J. Cao]{Junyan Cao}

\address{Universit\'e C\^ote d'Azur, CNRS, LJAD, France}

\email{junyan.cao@unice.fr}

\author[S. Dumitrescu]{Sorin Dumitrescu}

\address{Universit\'e C\^ote d'Azur, CNRS, LJAD, France}

\email{dumitres@unice.fr}

\author[H. Guenancia]{Henri Guenancia}

\address{Institut de Math\'ematiques de Toulouse; UMR 5219, Universit\'e de Toulouse; CNRS, 
UPS, 118 route de Narbonne, F-31062 Toulouse, France}

\email{henri.guenancia@math.cnrs.fr}



\date{}

\begin{abstract}
Let $X$ be a compact complex manifold such that its canonical bundle $K_X$ is numerically 
trivial. Assume, additionally, that $X$ is either Moishezon or $X$ is Fujiki with dimension at most 
four. Using the MMP and classical results in foliation theory, we prove a 
Beauville--Bogomolov type decomposition theorem for $X$. We deduce that holomorphic 
geometric structures of affine type on $X$ are in fact locally homogeneous away from an 
analytic subset of complex codimension at least two, and that they cannot be rigid unless 
$X$ is an \'etale quotient of a compact complex torus. Moreover, we establish a 
characterization of torus quotients using the vanishing of the first two Chern classes which 
is valid for any compact complex $n$-folds of algebraic dimension at least $n-1$. Finally, 
we show that a compact complex manifold with trivial canonical bundle bearing a rigid 
geometric structure must have infinite fundamental group if either $X$ is Fujiki, or $X$ is 
a threefold, or $X$ is of algebraic dimension at most one.
\end{abstract}

\maketitle

\tableofcontents

\section{Introduction}

The topic of geometric structures on manifolds, especially the automorphism groups
of such structures, is classical. The fundamental works of several leading
mathematicians, such as C. F. Gauss, B. Riemann, F. Klein, S. Lie and E. Cartan, created the foundation of this field. In particular, 
E. Cartan introduced and studied what is now known as Cartan geometry. These are geometric structures infinitesimally modelled on the homogeneous spaces 
\cite{Sh}. These geometric structures are flat when they are actually locally modelled (not just
infinitesimally) on the homogeneous spaces \cite{Sh}.

Important new results pertaining to the partition of a geometric manifold into orbits of local automorphisms of the geometric
structure were obtained by Gromov in \cite{Gr} (see also the elegant expository work \cite{DG}). In \cite{Gr} Gromov introduced the
{\it rigid geometric structures} (see Definition \ref{def-a}) as a broad class of geometric structure for which a (local)
automorphism is completely determined by its finite order jet at any given point. Affine and projective connections on the tangent bundle, 
pseudo-Riemannian metrics and conformal structures of dimension at least three are important examples of rigid geometric structures 
(they are also examples of Cartan geometries). On the other hand, symplectic structures and foliations are not rigid.

Earlier works, \cite{BD,BD2, BD3,D1,D2, BDG}, which were inspired by \cite{Gr,DG}, aimed to adapt Gromov's ideas and arguments 
to holomorphic geometric structures on compact complex manifolds. In that vein, the third-named author proved the following theorem: 

\begin{theorem}[{\cite{D2}}]
\label{thm sorin}
Let $X$ be a compact K\"ahler manifold $X$ with trivial first Chern class bearing a holomorphic geometric structure $\phi$ of affine type. Then
\begin{enumerate}[label=(\roman*)]
\item $\phi$ is locally homogeneous. 
\item If $\phi$ is rigid, then $X$ is covered by a compact complex torus. 
\end{enumerate}
\end{theorem}

It is important to keep in mind that the proof of the first statement relies in an essential way on the Bochner principle, while the 
second statement uses the Beauville--Bogomolov decomposition theorem \cite{Be,Bo1}.

\subsection{A Beauville--Bogomolov decomposition theorem for Moishezon manifolds}

The goal of the present paper is to generalize Theorem~\ref{thm sorin} to compact complex manifolds with trivial first Chern class that 
are not necessarily K\"ahler. A distinguished class of such manifolds is provided by compact Fujiki
manifolds; see Section~\ref{sec fujiki} for explicit non-K\"ahler examples. At the moment, there is 
no Bochner principle available in full generality, and a structural result in the spirit of the Beauville--Bogomolov theorem is 
unknown as well (see Conjecture~\ref{conj structure}). In short, it is expected that compact Fujiki manifolds with trivial canonical bundle are made up from irreducible K\"ahler varieties (i.e., tori, Calabi--Yau or symplectic holomorphic) with at most terminal singularities using 
\begin{enumerate}[label=$\bullet$]
\item small modifications,
\item products,
\item finite \'etale quotients.
\end{enumerate}

Our first result gives a partial answer to the above expectation. 

\begin{bigthm}\label{thma}
Let $X$ be a compact Fujiki manifold such that $c_1(X)=0\,\in\, {\rm H}^2(X, \,\R)$. Assume that one of the following holds:
\begin{enumerate}[label= $\circ$]
\item $X$ is Moishezon, or
\item $\dim X \,\le\, 4$.
\end{enumerate}
Then, there exists a finite \'etale cover $X'\,\longrightarrow\, X$ and a decomposition
\[X'\simeq T\times \prod_{i\in I} Y_i \times \prod_{j\in J} Z_j\]
where $T$ is a compact complex torus, the $Y_i$'s are irreducible Calabi--Yau manifolds and the $Z_j$'s are irreducible holomorphic symplectic manifolds. 

Moreover, each factor $Y_i$ (respectively, $Z_j$) in the decomposition is bimeromorphic to a K\"ahler variety with terminal singularities which is irreducible Calabi--Yau (respectively, irreducible holomorphic symplectic).
\end{bigthm}

We refer to Definition~\ref{def ICYIHS} and the remarks below it for the definitions of irreducible Calabi--Yau manifolds and irreducible holomorphic symplectic manifolds, which
actually mimic the definition in the singular K\"ahler case provided in e.g. \cite{GGK, CGGN} and coincide with the usual definitions in the smooth K\"ahler case.

From the second part in the statement of Theorem~\ref{thma} and the properties K\"ahler ICY and IHS varieties \cite{GGK, CGGN}, we deduce 
\begin{enumerate}[label=$\bullet$]
\item a Bochner principle for holomorphic tensors on $X$, 
\item a polystability result for $TX$ with respect to some movable classes, and
\item finiteness results for the linear part of the fundamental group of $X$;
\end{enumerate}
see Theorem~\ref{polystability} and Corollary~\ref{fund group}. 

In the Moishezon case, we provide alternative and purely algebraic arguments to study the
semistability of the tangent bundle (see Section~\ref{semistability} and Proposition~\ref{prepop}). 

An important application of the Bochner principle is provided by the following partial generalization of the first item of Theorem~\ref{thm sorin} in the case of Moishezon manifolds or Fujiki manifolds with dimension at most four. 

\begin{bigcor}
Let $X$ be as in Theorem~\ref{thma}.
Then there exists a Zariski open subset $U \,\subset\, X$, whose complement has complex codimension at least two, such that any holomorphic geometric structure $\phi$ of affine type on $X$ is locally homogeneous on $U$.
\end{bigcor}

As an easy application, we show that if a compact Fujiki manifold $X$ with trivial first Chern class bears a rigid holomorphic geometric structure, then $\pi_1(X)$
must be infinite (see Theorem~\ref{thmFujiki}).\\

{\it A few words on the proof of Theorem~\ref{thma}}.

\noindent
First, $X$ admits a projective/K\"ahler bimeromorphic model which 
itself admits a singular minimal model $X_{\rm min}$ in that same category by results in the Minimal Model Program (\cite{BCHM} 
and \cite{Dru11} in the projective case, and \cite{HP1} and \cite{DHP} in the K\"ahler case in dimension at most four). Next, an easy application of the negativity lemma shows that $X_{\rm min}$ has torsion canonical bundle, so that one can apply to $X_{\rm min}$ the decomposition theorem proved in \cite{HP2} and \cite{BGL} in the projective and K\"ahler case respectively. From there, one can show that the polystable decomposition of $TX$ induces regular foliations with compact leaves. Using Reeb stability theorem and the Barlet space of cycles on $X$, one can then obtain the product structure on $X$.

If one could prove that a compact K\"ahler manifold with zero numerical dimension has a minimal model, then Theorem~\ref{thma} would be valid for any compact Fujiki manifold with vanishing first Chern class. 

\subsection{Uniformization by compact complex tori}

Invoking Theorem~\ref{thma}, we give several characterizations of compact complex manifolds covered by a compact complex torus, partly in the spirit of the second item in Theorem~\ref{thm sorin}.

\begin{bigthm}
\label{thmb}
Let $X$ be a compact complex manifold of dimension $n$ such that $c_1(X)\,=\,0\,\in\, {\rm H}^2(X, \,\R)$, and denote by $a(X)$ 
the algebraic dimension of $X$. Assume that one of the following holds:
\begin{enumerate}[label= $\circ$]
\item $a(X) \,\ge \,n-1$ and $c_2(X)\,=\,0\,\in\, H^4(X,\,\R)$, or
\item $X$ is either Moishezon or Fujiki of $\dim X \,\le\, 4$, and $X$ bears a rigid holomorphic geometric structure of affine type.
\end{enumerate}
Then there exists a finite \'etale cover $T\,\longrightarrow\,X$ where $T$ is a complex torus. 
\end{bigthm}

Let us comment on each case individually. 

$\bullet$ If $X$ is K\"ahler, then this uniformization result is a classical consequence of Yau's solution of the Calabi conjecture 
\cite{Ya}. This problem has also recently attracted a lot of attention in the singular projective or K\"ahler setting (see \cite{GKP16} 
or \cite{CGG}).

In the case where $X$ is Moishezon, then the conclusion of Theorem~\ref{thmb} follows from the polystability of $TX$ proved in Theorem~\ref{polystability} coupled with the results of Demailly-Peternell-Schneider \cite{DPS94} on hermitian flat manifolds
(see Corollary~\ref{c2 zero}). It should be mentioned again that Theorem~\ref{polystability} is deeply connected to the recent progress on our understanding of singular K\"ahler varieties with zero first Chern class \cite{HP2, BGL}. \\

$\bullet$ The case where $a(X)\,=\,n-1$ is treated in Corollary~\ref{c_2=0, Fujiki}. An important ingredient of the proof is a result of 
Lin \cite{Lin} which shows that $X$ can be deformed to a Moishezon manifold. In order to reduce our situation to the Moishezon case 
previously established, we need to prove that the Albanese map of $X$ is \'etale trivial, which we show to hold for any Fujiki manifold 
with trivial first Chern class (see Theorem~\ref{splittorus1}).

It should be mentioned that while compact K\"ahler manifolds with numerically trivial canonical bundle are known to admit algebraic 
approximation \cite{Ca}, it is not known whether all Fujiki class $\mathcal C$ manifolds with numerically trivial canonical bundle 
admit algebraic approximations.\\

$\bullet$ Finally, the case where $X$ is Moishezon or Fujiki of dimension at most four is showed in Theorem~\ref{moishezon}, and partially generalizes the second item in Theorem~\ref{thm sorin}. The 
assumptions are used in order to get a decomposition theorem for $X$ (see Theorem~\ref{polystability}). A crucial observation is that 
given a rigid structure on a manifold $Y$ --- which need not be simply connected --- one can globally extend local Killing vector fields as 
long as any linear representation $\pi_1(Y)\,\longrightarrow\, \mathrm{GL}(N,\mathbb C)$ is trivial (see Remark~\ref{rep pi1}). This 
observation allows us to apply the finiteness result established in \cite{GGK} for linear representations of the fundamental group of 
minimal models with vanishing augmented irregularity, see Corollary~\ref{fund group}.

\subsection{Rigid holomorphic structures and fundamental groups}

Our last result shows that in many instances, compact complex manifolds with trivial canonical bundle bearing a rigid holomorphic 
geometric structure have infinite fundamental group. The following result is a combination of Corollary~\ref{thmFujiki} and 
Corollary~\ref{thm part cases}.

\begin{bigthm}\label{main}
Let $X$ be a compact complex manifold $X$ with trivial canonical bundle $K_X$ bearing a rigid holomorphic geometric structure
of affine type. Assume that $X$ satisfies one of the following assumptions:
\begin{enumerate}[label=$\circ$]
\item $X$ is a Fujiki manifold, or

\item $\dim X\,=\, 3$, or

\item the algebraic dimension of $X$ is at most one.
\end{enumerate}
Then the fundamental group of $X$ is infinite.
\end{bigthm}

Theorem~\ref{main} also holds for the holomorphic projective connections, and also for the holomorphic conformal structures, even 
though these two geometric structures are not of affine type (see Definition \ref{def-a}). This is because on manifolds with trivial 
canonical bundle, these two geometric structures lift to global representatives which are of affine type, namely, a holomorphic affine 
connection and a holomorphic Riemannian metric respectively (see Proposition \ref{proj con}). It may be mentioned that the particular 
case of holomorphic Riemannian metrics was settled earlier in \cite{BD3}; however, the proof in \cite{BD3} is of very specific nature 
and it works only in that particular context.

On the other hand, Theorem \ref{main} does not hold in general for non-affine geometric structures. Indeed, 
according to Definition \ref{def-a}, holomorphic embeddings of a compact complex manifold in complex projective 
space ${\mathbb C}{\mathbb P}^N$ are rigid holomorphic geometric structures (of order $0$). So simply connected complex projective 
manifolds admit rigid holomorphic geometric structures (of order $0$) of non-affine type.

To give a more illuminating example of a rigid geometric structure of non-affine type which is not locally homogeneous, recall that 
a compact complex torus $T^n\,=\,\mathbb C^n / \Lambda$ (here $\Lambda$ is a cocompact lattice in
$\mathbb C^n$), which admits a nontrivial holomorphic map to an elliptic curve, can be endowed with a holomorphic foliation that 
is not translation invariant \cite{Gh}. This combined with the standard holomorphic parallelization of the holomorphic tangent bundle of $T^n$
produces a holomorphic {\it rigid} geometric structure of non-affine type in the sense of Gromov (see Definition \ref{def-a} and Definition 
\ref{d-rigid}) which is not locally homogeneous.

Also rational homogeneous manifolds $X$ are simply connected and admit holomorphic rigid geometric structures of affine 
type. Indeed, a basis of ${\rm H}^0(X,\, TX)$ defines such a structure on 
$X$ (see Definition \ref{d-rigid} and Remark \ref{rational}). In this case $TX$ is globally generated and $K_X$ is not trivial.

The strategy of the proof of Theorem \ref{main} is the following. 

$\bullet$ In the Fujiki case (Corollary~\ref{thmFujiki}), this is a consequence of Theorem~\ref{thma} (see the second remark below the statement of the theorem) since the case of non-maximal algebraic dimension had been treated before in \cite{BD2}. As an interesting consequence, we 
establish Corollary \ref{Kill rigid} asserting that on a compact complex manifold with trivial canonical bundle, any holomorphic rigid 
geometric structure of affine type $\phi$ admits non-zero locally Killing vector fields. Such vector fields can be chosen to be global if 
$X$ is simply connected.

$\bullet$ To address the non-Fujiki case, the key result that we prove is Theorem~\ref{main lemma}, which asserts that the 
automorphism group of $(X, \,\phi)$ contains a maximal abelian Lie subgroup $A$ whose orbits in $X$ are closed and coincide with the 
fibers of a holomorphic submersion $\pi \,: \,X \,\longrightarrow\, B$ over a compact simply connected Moishezon manifold with globally 
generated canonical bundle $K_B$. Moreover, the fibers of $\pi$ are compact complex tori and the family $\pi$ is \textit{not} 
isotrivial (or equivalently, $K_B$ is not holomorphically trivial and $A$ is non-compact).

The fibrations constructed in Theorem~\ref{main lemma} cannot exist if $X$ is a compact K\"ahler Calabi--Yau manifold and we conjecture 
that they should neither exist in the broader context when $M$ is a compact simply connected complex manifold with trivial canonical 
bundle (see Remark~\ref{end conj}).

\section{Holomorphic rigid geometric structures}\label{section: geometric structures}

\subsection{Definitions and examples}

In this section we recall the context, definitions as well as some basics about
the {\it rigid geometric structures} in the sense of Gromov \cite{DG, Gr}.

To fix notation, consider a complex manifold $X$ of (complex) dimension $n$. Given any 
integer $r \,\geq \,0$, associate to it the principal bundle of $r$--frames $R^r(X)
\,\longrightarrow\, X$, which is the bundle of $r$--jets of local holomorphic coordinates on 
$X$ (i.e., $r$--jets of {\it local} biholomorphisms from $\mathbb C^n$ to $X$). It is a 
holomorphic principal bundle over $X$ with structure group $D^r({\mathbb C}^n)$ (or simply 
$D^r$) which is the group of $r$--jets, at origin, of local biholomorphisms of $\mathbb{C}^n$ 
fixing the origin. Notice that $D^r$ is a complex affine algebraic group. Let us now recall 
a basic definition from \cite{DG, Gr}.

\begin{definition}\label{def-a}
A {\it holomorphic geometric structure} (of order $r$)
on the complex manifold $X$ is a holomorphic $D^r$--equivariant map $\phi \,:\, R^r(X)\,
\longrightarrow\, Z$, with $Z$ being a complex algebraic variety endowed with an
algebraic action of the above group $D^r$. The geometric structure $\phi$ is said
of \textit{affine type} if the complex variety $Z$ is actually affine.
\end{definition}

To give examples of holomorphic geometric structures, holomorphic maps from $X$ 
to a complex algebraic variety $Z$ are evidently holomorphic structures of order $0$; they are of 
affine type if $Z$ is affine. Holomorphic tensors on $X$ are holomorphic geometric 
structures on $X$ of affine type of order one. Holomorphic affine connections on the 
holomorphic tangent bundle $TX$ are holomorphic geometric structures of affine type on $X$ 
of order two \cite{Gr,DG}. Holomorphic fields of planes, holomorphic flags, holomorphic 
foliations, holomorphic projective connections and holomorphic conformal structures are 
all holomorphic geometric structures of non-affine type.

A local biholomorphism $f \,:\, U \,\longrightarrow\, V$ between two open subsets $U$ and 
$V$ of $X$ is called a {\it local automorphism} (or {\it local isometry}) with respect to
a holomorphic geometric structure $\phi$ on $X$ of order $r$
if its natural lift to a map between the corresponding frame bundles
$$f^{(r)}\, :\, R^r(U)\,\longrightarrow\, R^r(V)$$ 
takes each fiber of $\phi$ to a fiber of $\phi$. This is the natural notion of local symmetry which 
coincides with the usual one in each of the examples of geometric structures.

The natural notion of a linearized symmetry is the following.

A (local) holomorphic vector field defined on an open subset $U \,\subset\, X$ is called a 
{\it Killing vector field}, with respect to $\phi$, if its local flow acts on $U$
through local automorphisms (or local isometries).

The group ${\rm Aut}(X,\, \phi)$ of all global automorphisms (isometries) of $(X,\, \phi)$ 
is a complex Lie subgroup of the group ${\rm Aut}(X)$ of biholomorphisms of $X$. Its 
unique maximal connected subgroup ${\rm Aut}_0(X,\, \phi)$ is a complex Lie subgroup of 
the unique maximal connected subgroup ${\rm Aut}_0(X)$ of the group of biholomorphisms of 
$X$. The Lie algebra of ${\rm Aut}_0(X,\, \phi)$ is the finite dimensional Lie algebra
consisting of all globally defined holomorphic Killing vector fields with respect to $\phi$.

Let $s$ be a nonnegative integer. The $s$--jet of the geometric structure $\phi$ of
order $r$ is the
geometric structure of order $(r+s)$ on $X$ defined by the map
\begin{equation}\label{ps}
\phi^{(s)}\,\,:\,\, R^{r+s}(X)\,\,\longrightarrow\,\, Z^{(s)}
\end{equation}
given by $\phi$, where 
$Z^{(s)}$ is the variety of $s$--jets of holomorphic maps from ${\mathbb C}^n$ to $Z$.
We note that $Z^{(s)}$ is naturally endowed with an algebraic action of $D^{r+s}$ by
pre-composition \cite{Ben, DG, Gr}; recall that
$R^{r+s}(X)$ is a principal $D^{r+s}$--bundle over $X$. For these actions of $D^{r+s}$
on $Z^{(s)}$ and $R^{r+s}(X)$, the above map $\phi^{(s)}$ is $D^{r+s}$--equivariant.

The $(r+s)$--jet of a local biholomorphism of $X$ is called
an {\it isometric jet} of order $s$ (with 
respect to $\phi$) if its canonical lift to $R^{r+s}(X)$ takes any fiber of the
map $\phi^{(s)}$ in \eqref{ps} to some fiber of $\phi^{(s)}$. This is 
the natural definition of an isometry of order $s$ \cite{Ben, DG, Gr}.

\begin{definition}\label{d-rigid}
A holomorphic geometric structure $\phi$ is called {\it rigid} of order $l$ (in the sense of
Gromov) if any $(l+1)$--isometric jet of $\phi$ is uniquely determined by its underlying
$l$--jet. In other words, the forgetful map from $(l+1)$--isometric jets $l$ jets is injective.
\end{definition}

Holomorphic affine connections are rigid of order one in the sense of
Gromov (see \cite{Ben,DG,Gr}). This is because any local 
biholomorphism, that fixes a point and preserves a connection, is actually
linearized in local exponential coordinates around the fixed point, which means that
such local biholomorphisms are completely determined just by their
differential at the fixed point.

Holomorphic Riemannian metrics are rigid holomorphic geometric structures of order one. 
Holomorphic conformal structures for dimension at least three, as well as holomorphic 
projective connections, are rigid holomorphic geometric structures of order two. The 
holomorphic symplectic structures, and the holomorphic foliations, are examples of non-rigid 
geometric structures~\cite{Ben,DG, Gr}.

The orbits, in $X$, of local isometries of $\phi$ are locally closed, and moreover the 
holomorphic tangent space to a given orbit space at any point of the orbit is spanned by 
the local holomorphic Killing vector fields \cite{Ben, DG,Gr}. The sheaf of local Killing 
vector fields of a rigid holomorphic geometric structure is locally constant \cite{No,Am, DG,Gr} (see below a
direct consequence stated as Theorem \ref{extend}). Its stalk at 
any point is a finite dimensional Lie algebra which is known as the {\it Killing algebra} 
of $\phi$.

\subsection{Two important results}

In this section, we would like to recall two results about holomorphic geometric structures that we will use extensively in this 
article.

To state the next result, recall that a complex manifold $X$ satisfies the Bochner principle if any holomorphic tensor field on $X$ 
vanishing at one point is vanishing identically. For instance, a compact K\"ahler manifold $X$ with $c_1(X)\,=\,0$ satisfies the Bochner 
principle, but other classes of examples exist. The following result was proved by the third-named author in \cite[Lemma~3.2]{D2}:

\begin{lemma}[{\cite{D2}}]
\label{loc hom}
Let $X$ be a complex manifold satisfying the Bochner principle. Then any holomorphic geometric structure of
algebraic affine type on $X$ is locally homogeneous. 
\end{lemma}	

The next result is an extendibility type result for local Killing fields relative to a rigid geometric structure on a simply connected 
manifold. It was first proved by Nomizu, \cite{No}, in the case Killing vector fields for real analytic Riemannian metrics and then 
extended to any rigid geometric structure by \cite{Am, DG,Gr}.

\begin{theorem}[{\cite{No,Am, DG,Gr}}]\label{extend}
Let $X$ be a complex manifold bearing a rigid holomorphic geometric structure $\phi$ and let $\xi$ be a local Killing field for $\phi$. 
If $X$ is simply connected, then $\xi$ extends to a global holomorphic vector field on $X$.
\end{theorem}

\begin{remark}
\label{rep pi1}
The arguments in the proof of Theorem~\ref{extend} actually show that one can replace the assumption that $\pi_1(X)$ is trivial by the weaker property that any complex linear representation $\pi_1(X)\,\longrightarrow\, \mathrm{GL}(n,\mathbb C)$ is trivial, where $n=\dim X$. 
\end{remark}

\subsection{Orbits of isometries and algebraic reduction} 

Recall that the {\it algebraic dimension} $a(X)$ of a compact complex manifold $X$ is the 
degree of transcendence over $\mathbb C$ of the field of meromorphic functions ${\mathcal 
M}(X)$ on $X$. It is known that $a(X) \,\in\, \{0,\,\ldots,\, \dim X\}$ \cite{Ue}. By 
definition, two bimeromorphic compact complex manifolds have the same algebraic dimension.

Compact complex manifolds $X$ with maximal algebraic dimension ($a(X)\,=\,\dim X$) are 
called Moishezon manifolds \cite[p.~26, Definition~3.5]{Ue}. It is the class of manifolds 
for which the meromorphic functions separate points in general position. Moishezon studied 
them in \cite{Mo} and proved that each of them is birational to some smooth complex 
projective manifold \cite{Mo}, \cite[p.~26, Theorem~3.6]{Ue}.

More generally, a compact complex manifold is said to be in {\it Fujiki class 
$\mathcal C$} (or a {\it Fujiki manifold} for short) if it is the meromorphic image of a compact K\"ahler manifold.
A basic result of Varouchas says
that a compact complex manifold belongs to Fujiki class $\mathcal C$ if and
only if it is bimeromorphic to a compact K{\"a}hler manifold \cite[Section\,IV.3]{Va}.
Manifolds lying in Fujiki class $\mathcal C$ share some of the
features of compact K\"ahler manifolds, e.g. the existence of a Hodge decomposition.

The following classical
result is known as the {\it algebraic reduction} theorem.

\begin{theorem}[{\cite[p.~25, Definition~3.3]{Ue}, \cite[p.~26, Proposition~3.4]{Ue}}]
\label{thue}
Let $X$ be a compact connected complex manifold of dimension $n$ and algebraic dimension
$a(X)\,=\,n-d$. Then there exists a bi-meromorphic modification $$\Psi \,:\, \widetilde{X}\,
\longrightarrow\, X$$ 
with $\widetilde X$ smooth and a holomorphic surjective map
$$t \,:\, \widetilde{X}\,\longrightarrow\, V\, ,$$ with connected fibers,
where $V$ is a $(n-d)$--dimensional projective manifold, such
that
\begin{equation}\label{h1}
t^* ({\mathcal M}(V))\,=\, \Psi^*({\mathcal M} (X))
\end{equation}
as subspaces of ${\mathcal M}(\widetilde{X})$.
\end{theorem} 

Consider the meromorphic fibration
\begin{equation}\label{pr}
\pi_{red}\,\,:\,\, X \,\longrightarrow\, V
\end{equation}
given by 
$t \circ \Psi^{-1}$ in Theorem \ref{thue}; it is called the algebraic reduction of $X$. If
the algebraic dimension of $X$ is zero, then the target of this algebraic reduction is a
point. Note that for manifolds with maximal algebraic dimension
Theorem \ref{thue} is equivalent to the earlier mentioned theorem of Moishezon.

Theorem 2.1 in \cite{D1} asserts that the fibers of the algebraic reduction of $X$ are 
contained in the orbits of local isometries of any given holomorphic rigid geometric structure $\phi$ (see also Theorem 3 in \cite{D3}). In
the special case where $X$ is 
simply connected, the following result is a direct consequence of Theorem 2.1 in 
\cite{D1}.

\begin{theorem}\label{act alg red}
Let $X$ be a compact complex simply connected manifold equipped with a holomorphic rigid 
geometric structure $\phi$. Then there exists
\begin{itemize}
\item an open dense subset $U\, \subset\, X$ whose complement $X \setminus U$ is
an analytic subspace of $X$ and $\pi_{red}$ (see \eqref{pr}) is defined on $U$, and

\item a connected complex abelian 
Lie subgroup $A$ of the automorphisms group ${\rm Aut}(X, \,\phi)$ of $(X,\, \phi)$,
\end{itemize}
such that the fiber of $\pi_{red}$ through any $z\, \in\, U$
is contained in the $A$--orbit of $z$. Moreover, 
$A$ is a maximal connected abelian subgroup in ${\rm Aut}(X, \,\phi)$. Furthermore, $A$
coincides with the identity component of the automorphism group ${\rm 
Aut}(X,\, \phi')$ of the rigid geometric structure $\phi'$ which is constructed as the 
juxtaposition of $\phi$ with a basis of the vector subspace of ${\rm H}^0(X,\, TX)$ 
spanned by the globally defined holomorphic Killing vector fields for $\phi$.
\end{theorem}

\begin{proof}
Given the extendibility result Theorem~\ref{extend}, Theorem 2.1 in \cite{D1} (see also Theorem 3 in \cite{D3}) implies that 
there is a open subset $U\, \subset\, X$ as in the statement of the theorem such that the 
fiber of $\pi_{red}$ through any $z\, \in\, U$ is contained in the ${\rm Aut}_0(X,\, 
\phi)$--orbit of $z$.

Fix a basis $\{X_1,\, \ldots,\, X_{k}\}\,\subset\, {\rm H}^0(X,\, TX)$ of the Lie
algebra of ${\rm Aut}_0(X, \,\phi)$. 
Let $\phi'$ denote the holomorphic rigid geometric structure on $X$ obtained by juxtaposing
$\phi$ with this family of holomorphic vector fields $\{X_1,\, \ldots,\, X_k\}$.

Denote by $A\,= \,{\rm Aut}_0(X,\, \phi')$ the connected component, containing the identity
element, of the automorphism group ${\rm Aut}(X,\, \phi')$ of $\phi'$. By construction, $A$
is a maximal connected abelian subgroup in ${\rm Aut}(X,\, \phi)$.
\end{proof}

The following proposition is proved using Theorem~\ref{extend}.

\begin{proposition}\label{loc hom2}
Let $X$ be a compact complex manifold with trivial canonical bundle. If $X$ admits a 
holomorphic rigid geometric structure which is locally homogeneous on an open dense 
subset, then the fundamental group of $X$ is infinite.
\end{proposition}

\begin{proof}
Let $\phi$ be a holomorphic rigid geometric structure on $X$ satisfying the condition of
being locally homogeneous on an open dense subset $U$. To prove the proposition by 
contradiction, assume that the fundamental group of $X$ is finite. Replace $X$ by its 
universal cover, and also replace $\phi$ by its pull-back through this finite covering 
map. Note that this pull-back of $\phi$ is again locally homogeneous on the
inverse image of $U$ (which is also an open dense subset). Therefore, we can assume that $X$ is 
simply connected and compact, and $\phi$ is locally homogeneous on an open dense subset.

Denote by $n$ the complex dimension of $X$. Since $\phi$ is locally homogeneous on an open 
dense subset of $X$, we may choose local holomorphic Killing vector fields $X_1,\, 
\ldots,\, X_n$ that span $TX$ over a nonempty open subset of $X$. By Theorem~\ref{extend}, the vector fields $X_i$ extend as global 
holomorphic sections of $TX$ on $X$. When contracted by $X_1\bigwedge \ldots\bigwedge 
X_n$, a nontrivial holomorphic section $\omega\,\in\, {\rm H}^0(X,\,K_X)$ defines a 
holomorphic function on $X$. This function must be constant, by the maximum principle, and 
it is nonzero at the points where the vector fields $X_1,\, \ldots,\, X_n$ are linearly 
independent. Consequently, this constant function is nonzero, which in turn implies that 
the vector fields $X_i$ are linearly independent at every point of $X$. Hence
the vector fields $X_1,\, \ldots,\, X_n$ span $TX$ on the whole of $X$.

In particular, the holomorphic tangent bundle of $X$ is trivial, in other words, $X$ is a 
parallelizable manifold, and hence it is biholomorphic to the quotient of a complex Lie 
group $G$ by a cocompact lattice of it \cite{Wa}. Consequently, the fundamental group of 
$X$ is infinite, which contradicts the fact that $X$ is simply connected.
\end{proof}

The following classical lemma will be useful in the proof of Theorem \ref{main}; its proof can be found in \cite[Lemma~2.5]{BD2}. 

\begin{lemma}\label{vol}
Let $X$ be a compact complex manifold with trivial canonical bundle $K_X$, and let $A$ be a 
connected complex Lie group acting on it through complex automorphisms. Then the choice
of any $\omega \,\in \,{\rm H}^0(X,\, K_X) \setminus\{0\}$ produces a smooth finite
$A$--invariant measure on $X$.
\end{lemma}

\section{Moishezon manifolds with vanishing first Chern class}\label{Section c1 zero Moishezon}

Recall that the holomorphic tangent bundle of a compact K\"ahler manifold, with numerically 
trivial canonical bundle, is polystable, this being an immediate consequence of Calabi's 
conjecture proved by Yau \cite{Ya}. The aim in this section is to prove a similar result in the 
set-up of Moishezon manifolds.

\subsection{Semistability of the tangent bundle} 
\label{semistability}

The main result of this section, Proposition~\ref{prepop}, states a strong semistability property for the tangent bundle of Moishezon manifolds with trivial first Chern class. Its proof is purely algebraic and relies on a deep result of Campana and P\u{a}un \cite{CP} recalled below. In \textsection~\ref{sec fujiki}, we will prove a strengthened result using (independent) analytic methods, but that only applies to {\it some} movable classes unlike Proposition~\ref{prepop}. 

\begin{theorem}[{\cite[Theorem 1.1]{CP}}]\label{algfoli}
Let $X$ be a projective manifold, and let $\cF\,\subset\, TX$ be a holomorphic foliation. Let $\alpha
\,\in\, \mathrm{Mov}_{\rm NS}(X)$ be a movable class such that the minimal slope $\mu_{\alpha, \rm min}(\cF)
\,>\,0$, i.e., for any quotient $\cF\,\longrightarrow\, \mathcal{G}$, the slope of $\mathcal{G}$
with respect to $\alpha$ is strictly positive. Then $\cF$ is algebraically integrable and a general
leaf of it is rationally connected.
\end{theorem}	

Recall that if $X$ is a compact Fujiki manifold of dimension $n$, then the movable cone ${\rm Mov}(X)\,\subset\,
H^{n-1,n-1}_{\rm BC}(X,
\, \mathbb R)$ is defined to be the closed convex cone generated by all the Bott--Chern classes of the form
$[f_*(\om_1\wedge \ldots \wedge \om_{n-1})]_{\rm BC}$ where $f\,:\,Y\,\longrightarrow\, X$ ranges over all K\"ahler
modifications of $X$ and $\om_1,\, \ldots,\, \om_{n-1}$ range over all K\"ahler metrics in $Y$. 

When $X$ is Moishezon, one can further define the cone of movable curves ${\rm Mov}_{\rm NS}(X)\,\subset
\,{\rm N}_1(X)_{\mathbb R}$ as the closed convex cone generated by the numerical classes of curves of the form
$[f_*(A_1\cap \cdots \cap A_{n-1})]$ where $f\,:\,Y\,\longrightarrow\, X$ ranges over all projective modifications of
$X$ and $A_1,\, \ldots,\, A_{n-1}$ range over all ample divisors on $Y$. \\

The following proposition shows that Theorem \ref{algfoli} continues to holds
when $X$ is Moishezon.

\begin{proposition}\label{algfoli2}
Theorem~\ref{algfoli} remains true under the weaker assumption that $X$ is a Moishezon manifold.
\end{proposition}

\begin{proof}
Let $f\,:\,X'\,\longrightarrow\, X$ be a projective modification and let $\cF'$ be the induced foliation on $X'$ constructed as follows. Let $U\subset X'$ be an open set over which $f$ is an isomorphism onto its image. Then $\cF'$ is the unique foliation on $X'$ such that $\cF'|_U= (f|_U)^*\cF$. Equivalently, $\cF'$ is the saturation of $(f|_U)^*\cF$ in $TX'$. 

Take $\alpha$ as in the statement of Theorem \ref{algfoli}. Then $f^*\alpha\,\in\, {\rm 
N}_1(X')_{\mathbb R}$ defined by $(D\cdot f^*\alpha)\,:=\,(f_*D\cdot \alpha)$, for any $D\,\in\, 
{\rm N}^1(X)$, is a movable class on $X'$. For any subsheaf $\mathcal E'\, \subset\,\mathcal F'$, 
let $\mathcal E$ be the saturation of $(f|_U)_*\mathcal E'$ in $TX$. We have 
$\mu_{\alpha}(\mathcal E)\,=\,\mu_{f^*\alpha}(\mathcal E')$, because $f^*\mathcal E$ and 
$\mathcal E'$ are isomorphic outside of the exceptional divisor of $f$. This shows that 
$\mathcal F'$ satisfies the assumptions of Theorem~\ref{algfoli}. Therefore, the leaves of
$\mathcal F'$ are algebraic, and a general leaf of it is rationally connected.

If $x\,\in\, X$ is a general point, the leaf $F_x$ of $\mathcal F$ through $x$ is the image under
$f$ of the leaf $F'_x \,\subset\, X'$ of $\mathcal F'$ through the point $f^{-1}(x)$. In particular, $F_x$ 
is open in its Zariski closure $\overline{F}_x \,\subset\, f(\overline {F_x'})$, which is rationally 
connected. This concludes the proof of the proposition.
\end{proof}

Now we will prove the semistability of $TX$ by using Proposition \ref{algfoli2}.

\begin{proposition}
\label{prepop}
Let $X$ be a Moishezon manifold with $c_1 (X)\,=\,0$. Let $\alpha\,\in\, {\rm Mov}_{\rm NS}(X)$ be a movable class.
Then the following two assertions hold:
\begin{enumerate}
\item The tangent bundle $TX$ is semistable with respect to $\alpha$.
		
\item Take $\alpha\,\in\, {\rm Mov}_{\rm NS}(X)^\circ$, and consider a filtration of $TX$ such that
the successive quotients are torsionfree and stable of same slope with respect to $\alpha$:
$$0\,=\, \cF_0 \,\subsetneq\, \cF_1\,\subsetneq\, \cF_2\,\subsetneq\, \cdots \,\subsetneq\, \cF_r\,=\, TX$$
(from the first statement it follows that such a filtration exists).
Then $c_1 (\cF_{i+1} /\cF_i)\,=\,0$ for every $0\, \leq\, i\, \leq\, r-1$.
\end{enumerate}
\end{proposition}

\begin{proof}
Assume that $TX$ is not semistable with respect to $\alpha$. Take $\cF_1\,\subset\, TX$ to be the maximal
semistable subsheaf, meaning the first nonzero term in the Harder--Narasimhan filtration of $TX$ with respect to $\alpha$.
Then \[c_1 (\cF_1) \cdot \alpha\,>\, \frac{\rank \cF_1}{n} c_1 (TX) \cdot \alpha\,=\,0.\] 

We first show that $\cF_1$ is a foliation, following \cite{Pe}. There is a natural map given by Lie bracket
\begin{equation}\label{e-1}
\bigwedge\nolimits^2 \cF_1 \,\longrightarrow \,[\cF_1,\, \cF_1] \,\longrightarrow \,TX / \cF_1
\end{equation}
which is $\mathcal O_X$--linear, and let $\cG\subset TX / \cF_1$ be the image. Assume by contradiction that $\cG$ is nonzero. 

Since the slope of $\cF_1$ is maximal among all subsheaves of $TX$, we easily see that 
$$\mu_{\alpha}(\cG)<0$$ 
by considering the exact sequence
$$0\,\longrightarrow\, \cF_1\,\longrightarrow\, \overline \cG\,\longrightarrow\, \cG
\,\longrightarrow\, 0,$$ where $\overline \cG\subset TX$ is the inverse image of $\cG$
by the quotient map $TX\,\longrightarrow\, TX/\cF_1$. Since $\cF_1$ is semistable, $\bigwedge\nolimits^2 \cF_1$ is semistable with respect to
$\alpha$ as well and $\mu_\alpha (\bigwedge\nolimits^2 \cF_1)\,=\,2 \mu_\alpha ( \cF_1)$. By semistability, we get
$$\mu(\cG)\ge 2 \mu_\alpha ( \cF_1)>0,$$ 
a contradiction. 
Consequently, we have $[\cF_1,\, \cF_1] \,\subset\, \cF_1$, and hence $\cF_1$ is a foliation.

Proposition~\ref{algfoli2} says that $\cF_1$ is algebraically integrable and a general leaf of it is rationally connected. In
particular, $X$ is uniruled. On the other hand, if $\pi\,:\, X'\,\longrightarrow\, X$ is a modification such that $X'$ is projective,
then $c_1(K_{X'})$ is effective, and hence $X'$ can't be uniruled for trivial reasons. In view of
this contradiction we conclude that $TX$ is semistable with respect to $\alpha$.
\medskip

Proof of the second statement of the proposition: We will prove using induction that
$c_1 (\cF_{i+1}/\cF_i) \,=\,0$ for all $0\, \leq\, i\, \leq\, r-1$.

The induction hypothesis says that $c_1 (\cF_{i+1}/\cF_i) \,=\,0$ for every $0\, \leq\, i \,<\, k$. 
By the first part of the proposition we know that 
\begin{equation}\label{zeroslope}
\mu_\alpha (\cF_{i+1}/\cF_i) \,=\, 0 
\end{equation} 
for every $0\,\leq \, i \,\leq\, r-1$. Take $\beta \,\in \,{\rm H}^{n-1, n-1} (X)\bigcap
{\rm H}^{2n-2} (X,\, \mathbb Q)$
close to $\alpha$. Since $\alpha$ is in the interior of the movable cone, it follows that $\beta$ is still movable. 
By applying the first part of the proposition to $\beta$ it is deduced that $TX$ is semistable with respect to $\beta$. Then
$\mu_\beta (\cF_{k+1}) \,\leq\, 0$. This combined with the induction hypothesis gives that
$$\mu_\beta (\cF_{k+1}/\cF_k)\, \leq\, 0 .$$
Since this holds for every $\beta$ close to $\alpha$, using \eqref{zeroslope} it follows that 
$c_1 (\cF_{k+1}/\cF_k) \,=\,0$.
\end{proof}

\subsection{Fujiki manifolds with vanishing first Chern class}
\label{sec fujiki}

In light of the Beauville--Bogomolov decomposition theorem for K\"ahler manifolds and its generalization to 
singular varieties (\cite{Dr2, GGK, HP2, BGL}), it is natural to introduce the following definition.

\begin{definition}
\label{def ICYIHS}
Let $X$ be a compact Fujiki manifold of dimension $n$ such that $K_X$ is linearly trivial. One says that $X$ is an
\begin{enumerate}[label=$\bullet$]
\item irreducible Calabi--Yau manifold (ICY) if for any finite \'etale cover $f\,:\,Y\,\longrightarrow\, X$ and any integer $0\,<\,p\,<\,n$, one has $H^0(Y,\, \Omega_Y^p)\,=\,\{0\}$. 
\item irreducible holomorphic symplectic manifold (IHS) if there exists a holomorphic symplectic form $\sigma
\,\in\, H^0(X,\, \Omega_X^2)$ such that for any finite \'etale cover $f\,:\,Y\,\longrightarrow\,
 X$, one has $\bigoplus_{p=0}^n H^0(Y,\, \Omega^p_Y) \,\simeq\, \mathbb C [f^*\sigma]$ as $\mathbb C$-algebras.
\end{enumerate}
\end{definition}

Let us make a few remarks about the above definition. 

$\bullet$ If $X$ is K\"ahler, an ICY or IHS manifold is automatically simply connected (as follows from the 
Beauville--Bogomolov decomposition theorem given that $\chi(X,\mathcal O_X)$ is respectively $2$ or $\frac n2+1$), hence our definitions are consistent with the ones existing in 
the K\"ahler case already. We expect that the same holds in the Fujiki setting, but we do not know how to 
prove it. We refer to Corollary~\ref{fund group} for a partial result in that direction.

$\bullet$ Definition~\ref{def ICYIHS} can be extended to the case where $X$ has at most canonical singularities, by replacing the sheaf of holomorphic forms $\Omega_X^p$ by that of reflexive holomorphic forms $\Omega_X^{[p]}$ and replacing \'etale covers by quasi-\'etale covers. We refer to \cite[Definition~6.11]{CGGN} for details. 

$\bullet$ If $X$ is a compact Fujiki manifold with $c_1(X)\,=\,0 \,\in\, H^2(X,\, \mathbb R)$, then $K_X$ is holomorphically torsion by \cite[Theorem~1.5]{To}. In particular, a finite \'etale cover of $X$ has trivial canonical bundle. 

$\bullet$ A Fujiki manifold $X$ which is ICY is automatically Moishezon. Indeed, any
K\"ahler modification $X'$ satisfies $h^0(X',\, \Omega_{X'}^2)\,=\,0$ and hence it is
projective by Kodaira's theorem.\\

Finally, let us introduce the following definition which will be important later.

\begin{definition}
\label{def q tilde}
 Let $X$ be a normal compact complex variety. The augmented irregularity of $X$ is defined as $$\widetilde 
q(X)\,\,:=\,\,\sup \{q(Y)\,\big\vert\, Y\,\longrightarrow\, X\, \, \mbox{ finite quasi \'etale} \} \,\in\, \mathbb N \cup \{\infty\},$$
where $q(Y)\,=\,h^0(Y, 
\, \Omega^{[1]}_{Y})=h^0(Y_{\rm reg}, \Omega^1_Y)$ is the irregularity, and $\Omega^{[1]}_{Y}$ denotes the sheaf of reflexive differential (i.e., the reflexive hull of the sheaf of K\"ahler differentials $\Omega^1_Y$).
\end{definition}

Recall that a finite cover $Y\,\longrightarrow\, X$ is called quasi-\'etale if it is \'etale in codimension one or, equivalently by Nagata's purity of branch locus, if it is \'etale over $X_{\rm reg}$.

In general, $\widetilde q(X)$ need not be finite. However, if $X$ is a Fujiki manifold of dimension $n$ such that $c_1(X)=0$, we will see that $\widetilde q(X)\le n$ as an application of (the proof of) Theorem~\ref{splittorus1}. Moreover, $\widetilde q(X)=n$ if, and only if $X$ is an \'etale torus quotient.

\begin{remark}
\label{aug irr codim 1}
If $X,Y$ are two normal compact complex varieties such that there exists a bimeromorphic map $X\dashrightarrow Y$ which is isomorphic in codimension one, then $\widetilde q(X)=\widetilde q(Y)$.
\end{remark}

We will now give standard examples of non-K\"ahler ICY and IHS manifolds. They are (small) modifications of 
K\"ahler ICY and IHS varieties with at most terminal singularities, and it is expected that they are all obtained in 
this fashion, see Conjecture~\ref{conj structure} below and the evidence for it provided by Theorem~\ref{polystability} in the Moishezon and 
low-dimensional cases. \\

{\bf Irreducible Calabi--Yau threefolds.}
A rich source of examples comes out of the ordinary double point
$V_n\,:=\,\{\sum_{i=0}^n z_i^2=0\} \,\subset\, \mathbb C^{n+1}$ for $n\ge 2$. The singularity $(V_n,0)$ is isolated, Gorenstein (as hypersurface singularity). If $n=2$, $V_2\simeq \mathbb C^2/\{\pm 1\}$ is a quotient singularity, but for $n\ge 3$, it is not anymore the case as one can see using Schlessinger's theorem asserting that isolated quotient singularities of codimension at least three are rigid. 

Alternatively, one can argue by observing that the local fundamental group of the singularity is trivial; that is $\pi_1(V_n\setminus \{0\})=\{1\}$. Let us sketch a proof of that fact which was communicated to us by the referee. Denote by $ Q$ the smooth projective quadric of dimension $n-1$. We have a $\mathbb C^*$-fiber bundle $\pi\,:\, V_n\setminus \{0\}\,
\longrightarrow\, Q$ which is nothing but the restriction of $L\,:=\,\mathcal O_{ Q}(-1)
\,\longrightarrow\, Q$ to the complement of the zero section. The fibration $\pi$ induces an exact sequence 
\[\pi_2( Q)\,\longrightarrow\, \mathbb Z\,=\,\pi_1(\mathbb C^*)
\longrightarrow\, \pi_1(V_n\setminus \{0\})\longrightarrow\, \pi_1( Q)\longrightarrow\, \{1\}.\]
Since $ Q$ is simply connected, all we have to do is to prove that the map $$\partial\,:\,\pi_2( 
Q)\,\longrightarrow\, \pi_1(\mathbb C^*)\,\simeq\, \mathbb Z$$ is onto. It is a classical fact (see e.g. 
\cite[p70-]{BT}) that $\partial$ corresponds to $c_1(L)\in H^2(Q,\,\mathbb Z)$ under the 
isomorphism $\pi_2( Q)\,\longrightarrow\, H_2(Q,\,\mathbb Z)$ provided by Hurewicz's theorem. Now Lefschetz 
hyperplane theorem implies that the natural map $H_2( Q,\,\mathbb Z)\,\longrightarrow\,
H_2(\mathbb P^n,\, 
\mathbb Z)$ is surjective since $n\,\ge\, 3$. Pairing with $c_1(\mathcal O_{\mathbb 
P^n}(-1))$ induces an isomorphism $H_2(\mathbb P^n,\, \mathbb Z)\,\longrightarrow\, \mathbb Z$; we deduce that 
$\partial$ is surjective, hence the result.

The ordinary double point can be seen as a cone over a smooth projective quadric $Q_n$. In particular, one can resolve the singular point by blowing up the origin once. Let $\pi\,:\,\widetilde V_n
\,\longrightarrow\,
V_n$ be the blow-up map. Then $K_{\widetilde V_n}\,=\,\pi^*K_{V_n}+(n-2)Q_{n-1}$,
and hence $V_n$ is canonical and terminal if $n\,\ge\, 3$. 

If $n=3$, small resolutions exist. It can be seen in many different ways. The quickest way is to rewrite $V_3\simeq \{xy=zt\} \subset \mathbb C^4$ and consider the graph $\widehat V_3 \subset V_3 \times \mathbb P^1$ of the meromorphic function $\frac xz=\frac ty$. It can be described explicitly as 
\[\widehat V_3 \,:=\,\{((x,\,y,\,z,\,t),\,[u:v])\,\in\, \mathbb C^4\times \mathbb P^1\,\,
\big\vert\,\, xy\,=\,zt,\, xv\,=\,zu,\,yv\,=\,tu\},\]
and the exceptional locus is simply $\{0\}\times \mathbb P^1$. Choosing the meromorphic function $\frac xt=\frac zy$ would have provided another small resolution. One recognizes the blow up of the two Weil (non Cartier) divisors $(x=z=0)$ and $(y=t=0)$. 

Alternatively, consider the blow up of the origin $\widetilde V_3$. It is isomorphic to the restriction $\mathbb 
L|_{Q_2}$ of the total space $\mathbb L$ of $\mathcal O_{\mathbb P^3}(-1)$ to the projective quadric $Q_2\simeq 
\mathbb P^1\times \mathbb P^1$. The normal bundle of the zero section $\mathbb P^3 \subset \mathbb L$ is 
isomorphic to $\mathcal O_{\mathbb P^3}(-1)$ hence the normal bundle of $Q_2\subset \widetilde V_3$ is isomorphic 
to $\mathcal O_{Q_2}(-1)$; therefore, its restriction to any ruling of the quadric is isomorphic to $\mathcal 
O_{\mathbb P^1}(-1)$. By the Nakano-Fujiki criterion, one can contract either family of $\mathbb P^1$s to a smooth 
space $\widehat V^3$, hence get a small resolution with exceptional locus isomorphic to $\mathbb P^1$. Contracting 
the other family of $\mathbb P^1$s provides another small resolution.

If $f\,:\,X\,\overset{2:1}{\longrightarrow}\, \mathbb P^3$ is a double cover of the projective
space ramified over a surface $B$, then $K_{X}\,=\,f^*(K_{\mathbb P^3}+\frac 12 B)$ is trivial
if and only if $\mathrm{deg}(B)\,=\,8$. In order to create singular examples, one will
consider octic surfaces $B$ with $s$ isolated singularities locally isomorphic
to $V_2$. The resulting $X$ will then have $s$ isolated singularities locally
isomorphic to $V_3$, which we can resolve by the procedure explained above, which glues
globally. We therefore get a small resolution $\pi\,:\,\widehat X\,\longrightarrow\, X$
with trivial canonical bundle (actually there are $2^s$ such resolutions). Such
varieties $X$ and $\widehat X$ have been extensively studied by Clemens \cite{Cle}. 

We claim that such an $\widehat X$ is an ICY threefold. Indeed, $\widehat X$ is simply connected \cite[Corollary~1.19]{Cle}, and for $p\in \{1,2\}$, one has successively
\begin{align*}
H^0(\widehat X,\, \Omega^p_{\widehat X})\,&\simeq\, H^0(X,\, \Omega_X^{[p]}) \\
\,&\simeq\, H^p(X,\,\mathcal O_X) \\
\,&\,\simeq H^p(\mathbb P^3,\, f_*\mathcal O_X)\\
\,& \simeq\, H^p(\mathbb P^3,\, \mathcal O_{\mathbb P^3}\oplus \mathcal O_{\mathbb P^3}(-4))
\end{align*}
which is indeed zero. The first two identities are consequences of \cite[Corollary~1.7]{KS} since $X$ has rational singularities; the third one comes from Leray spectral sequence given that $R^qf_*\mathcal 
O_X\,=\,0$ if $q\,>\,0$, since $f$ is finite. The fourth one is a classic result for cyclic covers, see for example 
\cite[2.50]{KM}. Alternatively, if one does not want to rely on the simple connectedness of $\widehat X$, 
one can use the decomposition theorem for $X$ \cite{HP2}. Indeed, since $X$ has isolated non quotient 
singularity, no finite quasi-\'etale cover of $X$ can split a torus; hence $X$ is irreducible. Since it is of 
odd dimension, it must be an ICY variety, hence so is $\widehat X$ which follows from
\cite[Theorem~1.4]{GKKP}, given 
that quasi-\'etale covers of $X$ are in one-to-one correspondence with \'etale covers of $\widehat X$, since 
$\widehat X \,\longrightarrow\, X$ is small.

Finally, one can find examples of such manifolds $\widehat X$ that are not projective algebraic, or equivalently K\"ahler. A great reference for that topic, which includes a survey of \cite{Cle} is the PhD thesis of J. Werner that has recently been translated into English in \cite{W}. Given any $B$ as before, it follows from Clemens results \cite[Theorem~4.3]{W} that one can find {\it one} small resolution $\widehat X\,\longrightarrow\, X$ such that the exceptional curves $C_1, \ldots, C_s$ satisfy a non-trivial relation $\sum_{i=1}^s m_i [C_i]=0$ in $H_2(\widehat X,\, \mathbb Q)$ with $m_i\,\ge\, 0$, which prevents $\widehat X$ from being K\"ahler. More precisely, the result cited guarantees that there exists a non-trivial relation between the $[C_i]$ without the information on the sign of the $m_i$, but then changing the small resolution of each ordinary double point has the effect of switching the sign of $m_i$. Now, by choosing $B$ to be certain specific Chmutov octic, one can see that {\it no} small resolution of the resulting double solid $X$ is K\"ahler, see \cite[\textsection\,6]{W}. \\

{\bf Irreducible holomorphic symplectic manifolds.}
The key construction here is the so-called Mukai flop. Let us briefly recall what it is and refer to e.g. \cite[Example~21.7]{GHJ} for more details. Let $X$ be an $2n$-dimensional K\"ahler (irreducible) holomorphic symplectic manifold containing a submanifold $P\subset X$ isomorphic to $\mathbb P^n$. The existence of a holomorphic symplectic form imposes $N_{P|X}\simeq \Omega_P^1$ and one can then see that the projective bundle $\mathbb P(\Omega_P^1)\,\longrightarrow\, P$ is isomorphic to the first projection of the incidence variety $D\subset \mathbb P^n \times (\mathbb P^n)^\vee$; i.e., $D\,=\,\{(x,\,H)\,\,\big\vert\,\, x\,\in\, H\}$. In particular, if $Z\,\longrightarrow\, X$ is the blow-up of $P$, then the exceptional divisor is isomorphic to $D$ and, moreover, the restriction of the normal $N_{D|Z}$ to any fiber of $D\,\longrightarrow\,
 (\mathbb P^n)^\vee$ is isomorphic to $\mathcal O_{\mathbb P^{n-1}}(-1)$ as one can see by using the adjunction formula twice (once for $D\subset \mathbb P^n \times (\mathbb P^n)^\vee$ and once for $D\subset Z$) . Using again Nakano-Fujiki's criterion, one can contract all such fibers in $Z$ to a obtain a smooth manifold $X'$
\begin{equation*}
\begin{tikzcd}
& \ar[dl,"p", swap]Z \ar[dr, "q"]& \\
X \ar[rr,dotted, "\phi"]& & X'
\end{tikzcd}
\end{equation*}
 which can then easily be shown to be holomorphic symplectic. 

The important property of the Mukai flop $X'$ of $X$ is as follows. Set $P'\,:=\,q(D)\,\simeq\, \mathbb P^n$. For dimensional reasons, there exists an isomorphism $H^2(X,\,\mathbb R)\,\longrightarrow\, H^2(X',\, \mathbb R)$. Let $\omega\in H^2(X,\,\mathbb R)$ and let $\omega'\,\in\, H^2(X',\, \mathbb R)$ be the corresponding class. We have $p^*\omega=q^*\omega'+a[D]$ for some $a\in \mathbb R$. Now, if $\ell \subset P$ and $H\subset P$ is an hyperplane containing $\ell$, then $\widetilde \ell:=\ell\times\{ H\}\subset D$ is a curve such that $p_*\widetilde \ell= \ell, q_*\widetilde \ell=0$ as cycles and $\mathcal O_Z(D)|_{\widetilde \ell} \simeq \mathcal O_{\mathbb P^1}(-1)$. In particular, $\omega \cdot \ell
\,=\, p^*\omega \cdot \widetilde \ell \,=\, -a$. One can do the same operation with a line $\ell'
\,\subset\, P'$ and find $\omega' \cdot \ell '\, =\,a$ so that 
\begin{equation}
\label{flop}
\omega \cdot \ell \,=\, -\omega' \cdot \ell'.
\end{equation}

We now give a general construction of non-K\"ahler IHS manifolds obtained as the Mukai flop of certain 
projective IHS manifolds as described in \cite[Example~21.9]{GHJ}. Assume that $X$ admits two {\it disjoint} 
submanifolds $P_1,\, P_2 \,\simeq\, \mathbb P^n$ and let $X'$ be the Mukai flop of $P_1$. The map 
$X\,\dashrightarrow \,X'$ is therefore isomorphic near $P_2$. Assume further that the Picard number of $X$ is 
two and that there exists a non-trivial morphism $f:X\,\longrightarrow\,
 Y$ to a projective variety $Y$ contracting both $P_1$ 
and $P_2$. If $H$ is an ample line bundle on $Y$ and $\ell_i$ is a line on $P_i$, then $\ell_i \in N_1(X)$ lies in the hyperplane $(f^*H)^{\perp}\subset N_1(X)_{\R}$. Since $\dim_\R N_1(X)_{\R}=2$, the latter is a line, hence $\mathbb R \ell_1\,=\, \mathbb R \ell_2\,\in\, H^{4n-2}(X,\,\mathbb R)$. If $\alpha$ is a K\"ahler class on $X$, we have $\alpha \cdot \ell_i>0$ for $i\in 
\{1,2\}$, hence
\begin{equation}
\label{positive}
\ell_1 \,=\, \lambda \ell_2 \quad \mbox{ for some } \, \lambda >0.
\end{equation}
If $X'$ were to admit a K\"ahler class $\omega'$, then the corresponding class $\omega\in H^2(X,\,\mathbb R)$ would satisfy
\[\omega \cdot \ell_1\,=\, -\omega' \cdot \ell'_1\,<\,0 \quad \mbox{and} \quad\omega \cdot \ell_2
\,=\,\omega'\cdot \ell_2 \,>\,0\]
(as follows from \eqref{flop}), which contradicts \eqref{positive}. Hence $X'$ cannot be K\"ahler. Explicit examples of IHS manifolds $X$ satisfying the requirements have been constructed by Yoshioka and
Namikawa; see references in \cite[Example~21.9]{GHJ}.

\subsection{The decomposition theorem}

We would like to propose the following statement on the structure of Fujiki manifolds with vanishing first 
Chern class, generalizing the well-known Beauville--Bogomolov decomposition theorem.

\begin{conj}[{Decomposition Conjecture}]
\label{conj structure}
Let $X$ be a compact Fujiki manifold such that $c_1(X)\,=\,0\,\in\, {\rm H}^2(X, \,\R)$. Then there exists a finite \'etale cover $X'
\,\longrightarrow\, X$ such that
\[X'\,\,\simeq\,\,T\times \prod_{i\in I} Y_i\ \times \prod_{j\in J} Z_j\]
where
\begin{enumerate}[label= $(\roman*)$]
\item $T$ is a complex torus,
\item For all $i\in I$, $Y_i$ an irreducible Calabi--Yau manifold,
\item For all $j\in J$, $Z_j$ is an irreducible holomorphic symplectic manifold.
\end{enumerate}
Moreover, there exist bimeromorphic maps $Y_i \dashrightarrow \widehat Y_i$ (respectively,
$Z_j \dashrightarrow \widehat Z_j$), isomorphic in codimension one, such that $ \widehat Y_i$ is a projective ICY variety with terminal singularities (respectively, $Z_j$ is a K\"ahler IHS variety with terminal singularities). 
\end{conj}

The goal of this section is to establish the above conjecture assuming either that $X$ is Moishezon or that it 
satisfies $\dim X\,\le\, 4$, which will enable us to refine the statements obtained in 
\textsection~\ref{semistability} above.

\begin{theorem}\label{polystability}
Let $X$ be a Fujiki manifold such that $0\,=\, c_1(X)\,\in\, {\rm H}^2(X, \,\R)$. Assume that one of the following holds:
\begin{enumerate}[label= $\bullet$]
\item $X$ is Moishezon, or
\item $\dim X \,\le\, 4$.
\end{enumerate}
Then the Decomposition Conjecture holds for $X$.

\noindent
Moreover, there exists a Zariski open set $U\,\subset\, X$,
whose complement $X\setminus U$ has codimension at least two, satisfying the
condition that there is an
incomplete Ricci-flat K\"ahler metric $\omega$ on $U$ for which the following two hold:
\begin{enumerate}[label= $(\roman*)$]
\item (Bochner principle).\, For every holomorphic tensor $\sigma \,\in\, {\rm H}^0(X,\, TX^{\otimes p} \otimes T^*X^{\otimes q})$
the restriction $\sigma\big\vert_U$ is parallel with respect to $\omega$.

\item (Polystability). There exists a subset $\Lambda \,\subset\, \mathrm{Mov}(X)$ with non-empty
interior such that $TX$ is polystable
with respect to any $\lambda \,\in\, \Lambda$. More precisely, there is a holomorphic decomposition of the tangent bundle
\begin{equation}\label{ed}
TX\,\,=\,\,\bigoplus_{i=1}^\ell F_i
\end{equation} 
satisfying the conditions that $F_i$ is locally free, $$c_1(F_i)\,=\,0\,\in \,{\rm H}^2(X, \,\R)$$
and $F_i$ is stable with respect to any $\lambda \,\in\, \Lambda$.
Additionally, over the subset $U$, the decomposition in \eqref{ed}
is orthogonal and $F_i\big\vert_{U}$ is parallel with respect to $\omega$. Finally, one can take $\Lambda\supset \mathrm{Mov}_{\rm NS}(X)^\circ$ if $X$ is Moishezon.
\end{enumerate}
\end{theorem}

\begin{proof}
We proceed in several steps. 

{\bf Step 1. Existence of a good minimal model. }

Let $X' \,\longrightarrow\, X$ be a modification such that $X'$ is K\"ahler. 
Since $K_X$ is torsion by \cite[Theorem 1.5]{To}, we know that $\kappa (X')\,=\,0$ and the numerical dimension ${\rm nd}(K_{X'})\,=\,0$. We have two cases: 
\begin{enumerate}[label=$\circ$]
\item If $X'$ is projective (i.e., if $X$ is Moishezon), then we can apply
\cite{BCHM} and in particular \cite[Corollary 3.4]{Dru11} to find a birational model $X'\dashrightarrow X_{\rm min}$ with terminal singularities and torsion canonical bundle $K_{X_{\rm min}}$.
\item If $X'$ is merely K\"ahler but $\dim X \le 4$, then if follows from \cite[Theorem~1.1]{HP1} in dimension three and \cite[Theorem~1.1]{DHP} in dimension four that there exists a bimeromorphic model $X'\dashrightarrow X_{\rm min}$ with terminal singularities and nef canonical bundle $K_{X_{\rm min}}$.
\end{enumerate}

\medskip
 Let us resolve the indeterminacies of $\phi\,:\, X\,\dashrightarrow\, X_{\rm min}$ as follows 
\begin{equation}\label{ep}
\begin{tikzcd}
& \ar[dl,"p", swap]Z \ar[dr, "q"]& \\
X \ar[rr,dashed, "\phi"]& & X_{\rm min}
\end{tikzcd}
\end{equation}
where $Z$ is smooth K\"ahler. 

We claim that the map $\phi$ is an isomorphism in codimension one and that $K_{X_{\rm min}}$ is torsion. Indeed, write
\[K_Z \,= \,p ^* K_X + \sum a_i E_i \ \ \,\, \text{ and }\ \ \,\, K_Z \,=\, q ^* K_{X_{\rm min}} + \sum b_j F _j ,\]
where $\sum E_i$ (respectively, $\sum F_j$) is the exceptional locus of $p$ (respectively, $q$). As $X$ and $X_{\rm min}$ are terminal, we
have $a_i,\, b_j \,>\,0$ for all indices $i,\,j$. Set $D\,:=\,\sum a_i E_i-\sum b_j F_j\,=\,
q^*K_{X_{\rm min}}-p^*K_X$. Then $-(-D)=D$ is $p$-nef and $p_*(-D)\,\ge\, 0$, so that by the negativity
lemma (see e.g. \cite[Lemma~1.3]{Wan}), we actually have $-D\,\ge\, 0$, i.e., $D\,\le\, 0$. Similarly, $-D$ is $q$-nef and $q_*(D)
\,\ge\, 0$ so $D\,\ge\, 0$. All in all, $D\,=\,0$. This implies that $q^*K_{X_{\rm min}}\,=\,
p^*K_X$ is torsion, hence so is $K_{X_{\rm min}}$. Moreover, since the coefficients $a_i,\,b_j$ are positive, then we
also get $\sum E_i\,=\,\sum F_j$ which implies that $\phi$ is an isomorphism in codimension one. Our claim follows. 

In the following, we denote by $U\,\subset\, X$ the maximal Zariski open subset over which $\phi$
is defined and is an isomorphism. By the above, 
the codimensions of $X\setminus\ U$ and $X_{\rm min} \setminus \phi(U)$ are at least two.

\medskip

{\bf Step 2. Bochner principle. }

Let $\alpha\in H^2(X,\,\R)$ be K\"ahler class on $X_{\rm min}$. By \cite{EGZ, Pa}, there exists a positive current 
$$\omega_{\rm min}\,\in\, \alpha$$ with bounded potentials on $X_{\rm min}$ whose 
restriction to the regular locus $X_{\rm min}^{\rm reg}$ of $X_{\rm min}$ is a genuine K\"ahler 
Ricci flat metric with finite volume equal to $\alpha^n$. Define
$$\omega\,:=\,\phi^*(\omega_{\rm min}\big\vert_{\phi(U)}),$$
where $\phi$ is as in \eqref{ep}; it is a K\"ahler Ricci flat metric on $U$. Since $\int_U \omega^n
\,<\,+\infty$, the K\"ahler manifold $(U,\omega)$ is incomplete unless $X\,=\,U$ is already K\"ahler (see \cite[Proposition~4.2]{GGK}).

By \cite[Theorem~A]{CGGN} every reflexive holomorphic tensor on $X_{\rm min}^{\rm reg}$
is parallel with respect to $\omega_{\rm min}$. Since $U$ and $\phi(U)$ have complements of codimension at least two, there are
natural isomorphisms
\[{\rm H}^0(U,\, TX^{\otimes a} \otimes T^*X^{\otimes b}) \,\simeq\, {\rm H}^0(X,\, TX^{\otimes a} \otimes T^*X^{\otimes b}) \]
and 
\[{\rm H}^0(\phi(U), \,TX_{\rm min}^{\otimes a} \otimes T^*X_{\rm min}^{\otimes b}) \,\simeq\,
{\rm H}^0(X_{\rm min}^{\rm reg},\, TX_{\rm min}^{\otimes a} \otimes T^*X_{\rm min}^{\otimes b})\]
for all integers $a,\,b \,\ge\, 0$. Statement $(i)$ follows immediately. 

\medskip

{\bf Step 3. Polystability of $TX$. }

The tangent bundle $TX_{\rm min}$ of $X_{\rm min}$ holomorphically decomposes as $$T_{X_{\rm 
min}}\,= \,\bigoplus_{i=1}^r F_i^{\rm min},$$ where each $F_i^{\rm min}$ is $\alpha$-stable with zero first Chern class 
and its restriction to the regular locus of $X_{\rm min}$ is a parallel subbundle with respect 
to $\omega_{\rm min}$ (see \cite[Proposition~D]{GGK} and \cite[Remark~3.5]{CGGN}). Using $\phi$, we get a similar 
holomorphic decomposition $$T_{X}\,= \,\bigoplus_{i=1}^r F_i$$ valid over $U$ for some parallel 
subbundles $F_i \,\subset\, T_U$ with respect to $\omega$. Taking the saturation of $F_i$ in 
$T_X$, we get reflexive subsheaves --- still denoted by $F_i$ --- and a homomorphism 
$\bigoplus_{i=1}^r F_i \,\longrightarrow\, T_X$ which is an isomorphism over $U$, hence it is an 
isomorphism everywhere.

Next, we check that $c_1(F_i)\,=\,0$. This is easy to see because 
\[p^*c_1(F_i)\,=\, q^*c_1(F_i^{\rm min})+\sum c_{ij} E_j\,=\, \sum c_{ij} E_j\]
for some $c_{ij}\,\in\, \mathbb Z$ (see \eqref{ep} for $p,\, q$). Recall that the divisors $E_j $ are exceptional for both
$p$ and $q$. In particular, for any $\gamma \,\in\, {\rm H}^{2n-2}(X)$ we 
have 
\[c_1(F_i) \cdot \gamma \,=\, p^*c_1(F_i) \cdot p^*\gamma\,=\,\sum c_{ij} \, E_j \cdot p^*\gamma \, = \, 0\] 
using the projection formula.

The same arguments show that $F_i$ is stable with respect to 
$$\beta\,:=\,p_*(q^*\alpha^{n-1})\,\in\, {\rm H}^{n-1, n-1}(X)\cap {\rm H}^{2n-2}(X,\, \R).$$ 
The class $\beta$ is movable for being a limit of the classes $p_*((q^*\alpha+\ep \gamma)^{n-1})$ as $\ep \,\to\, 0$, where
$\gamma \in H^2(Z,\,\R)$ is a fixed K\"ahler class. Finally, \cite[Remark~3.5]{GKP} shows that the set $\Lambda \,\subset\,
\mathrm{Mov}(X)$ of movable classes with respect to which $F_i$ is stable has non-empty interior for every $i\,=\,1, \,\ldots,\, r$. This 
together with the condition $c_1(F_i)\,=\,0$ implies that $TX$ is polystable with respect to any 
$\lambda\,\in\, \Lambda$.

It remains to see that if $X$ is Moishezon, then for any $\lambda \in \mathrm{Mov}_{\rm NS}(X)^\circ$ and 
any index $i$, $F_i$ is $\lambda$-stable. We know from Proposition~\ref{prepop}
that $F_i$ is $\lambda$-semistable.
Arguing by contradiction, one can consider the maximal destabilizing sheaf 
$G_i\subset F_i$. It is a proper, $\lambda$-stable subsheaf with $\mu_\lambda(G_i)=0$. Now, given any 
$\alpha \in {\rm N}_1(X)_{\mathbb R}$, we have $\lambda+\ep \alpha \in \mathrm{Mov}_{\rm NS}(X)$ for $0 \le 
\ep \ll 1$, hence $\mu_{\lambda+\ep \alpha}(G_i) \le 0$ by Proposition~\ref{prepop}. This implies that 
$c_1(G_i) \cdot \alpha \le 0$. Since this holds for any $\alpha \in {\rm N}_1(X)_{\mathbb R}$, we deduce 
that $c_1(G_i)=0 \in {\rm N}^1(X)_{\mathbb R}$. This is a contradiction with the fact that $F_i$ is stable 
with respect to some movable class and satisfies $c_1(F_i)=0$, hence it shows the claim.

\medskip

{\bf Step 4. The splitting of $X$. }

The validity of Bochner principle implies that the Albanese map of our manifold $X$ is \'etale trivial (see proof of 
\cite[Theorem~4.1]{CGGN} or Theorem~\ref{splittorus1} below). This implies that a finite \'etale cover $X'\,\longrightarrow\, X$ splits 
as a product $X'\,\simeq \,T\times Y$ where $T$ is a complex torus and $Y$ is a compact Fujiki manifold $Y$ such that $K_Y$ is trivial 
and $H^0(Y',\, \Omega_{Y'}^1)\,=\,0$ for any finite \'etale cover $Y'\,\longrightarrow\, Y$. Up to replacing $X$ with $Y$, one can 
assume that the augmented irregularity of $X$ vanishes. In particular, the augmented irregularity of $X_{\rm min}$ vanishes as well thanks to Remark \ref{aug irr codim 1}. Using the decomposition theorem for $X_{\rm min}$ (see \cite{HP2} or \cite{BGL}), one can find a finite quasi-\'etale cover $g\,:\,X'_{\rm min}\,\longrightarrow\, X_{\rm min}$ which 
further decomposes as a product \[X'_{\rm min}\,\,\simeq\,\, \prod_{i\in I} Z_i\] where the $Z_i$ are irreducible singular ICY or IHS varieties in the sense of {\it loc. cit.} The absence of a torus factor comes from the equality $\widetilde q(X_{\rm min})=0$. 

Set $V\,:=\,\phi(U)\subset X_{\rm min}^{\rm reg}$. The restriction $g|_{g^{-1}(V)}$ is \'etale and induces 
via $\phi$ a finite \'etale cover $U'\,\longrightarrow\, U$ that extends to a finite cover 
$f\,:\,X'\,\longrightarrow\, X$ for some compact normal space $X'$ thanks to \cite[Theorem~3.4]{DeGr}. Since $f$ is \'etale in codimension one and $X$ is smooth, the purity 
of branch locus ensures that $f$ is \'etale everywhere hence $X'$ is smooth. We therefore have the following diagram
\begin{equation}
\label{finite cover}
\begin{tikzcd}
X'\ar[r,"\psi", dotted] \ar[d, "f", swap] & X'_{\rm min} \ar[d,"g"] \\
X \ar[r,dotted, "\phi"]& X_{\rm min}
\end{tikzcd}
\end{equation}
{}From now on, we replace $X$ with $X'$. By induction, one can assume that $X_{\rm min}
\,\simeq\, Y\times Z$ where $Y$ and $Z$ are compact K\"ahler with terminal singularities of
respective dimensions $k$ and $\ell$, not necessarily irreducible (i.e., they can be a
product of lower dimensional ICY or IHS varieties). On $X$, we have an induced splitting of
the tangent bundle into regular foliations $T_X\,=\, \cF\oplus \cG$. Indeed, since $\cF$
and $\cG$ are direct summand of $T_X$, they are automatically subbundles of $T_X$. Moreover,
the image of the composition of maps $$[\cF, \cF]\,\longrightarrow\, T_X\,\longrightarrow\,
T_X/\cF\,\simeq\, \cG$$ is supported on the proper subset $X\setminus U$ and hence it vanishes
since $\cG$ is torsion-free. This shows that $\cF$ is a regular foliation, and
we can argue similarly for $\cG$. 

Next, we will show that $\cF$ has smooth compact leaves with finite holonomy. It is clear 
that a general leaf is compact, hence smooth since $\cF$ is regular. Indeed, if $x\,\in\, U$, 
then the leaf $F_x$ of $\cF$ though $x$ is included in the strict transform by $\phi$ of the 
leaf $Y\times \{z\}$ where $\phi(x)\,=\,(y,\,z)$. Now we want to show that if $F$ is a general 
(compact) leaf of $\cF$, then the volume $\int_F \omega^k$ of $F$ with respect to a fixed 
hermitian metric $\omega$ is bounded uniformly independently of $F$, which will show that 
all leaves of $\cF$ are compact by Bishop theorem; smoothness will then follow from the 
regularity of $\cF$.

In order to show that the volume is bounded, let $W$ be a K\"ahler desingularization of the graph of $\phi$, with projections
$p\,:\,W\,\longrightarrow\, X$ and $q\,:\,W\,\longrightarrow\, Y\times Z$. Set $$r\,:=\,\mathrm{pr}_Z \circ q
\,:\, W\,\longrightarrow\, Z$$ so that the cycles $W_z\,:=\,r^{-1}(z)$ are
homologous (and smooth) for $z$ general. When $z\,\in\,
Z$ varies (say in a dense open subset of $Z$), the varieties $F_z\,:=\,
p(W_z)$ form a family of smooth leaves of $\cF$ sweeping out a dense open set of $X$. Moreover, $p$
induces a bimeromorphic map between $W_z$ and $F_z$. In particular, we have $\int_{F_z} \omega^k
\,=\,\int_{W_z}p^*\omega^k$. Pick a K\"ahler metric $\omega_W$ such that $\omega_W \,\ge\,
p^*\omega$. Then we have $\int_{F_z} \omega^k \,\le\, \int_{W_z} \omega_W^k$ which is independent
of $z$ (general) since the $W_z$ are homologous and $\omega_W$ is closed. 

Finiteness of the holonomy can be shown as follows. Let $x\,\in\, X$ and let $F_x$ be the leaf of $\cF$ though $x$. Choose a transversal $S\simeq \mathbb D^{\ell}$ to $\cF$ at $x$ (e.g. a neighborhood of $x$ in the leaf of $\cG$ through $x$). Choose a chart $U_x$ near $x$ given by the unit polydisk where the leaves of $\cF$ are given by the affine subspaces $V_a:=(z_1\,=\,a_1,\,
\ldots,\, z_{n-k} \,=\,a_{\ell})$, $a\,\in\, \mathbb D^{\ell}$, and where $\omega
\,\ge\, C^{-1} \omega_{\rm eucl}$. If a leaf $L$ of $F$ hits $S$ at a point $a$, it means that $L\cap U_x$ contains $V_a$; in particular, if $L$ hits $S$ at least $N$ times, then the volume of $L$ with respect to $\omega$ is at least $NC^{-k}k! \pi^k$ hence 
 $N\le \frac{C^k}{k!\pi^k} \int_{F} \omega^k$ and that quantity is bounded above independently of $F$. 

Using the holomorphic version of Reeb stability theorem (see for example \cite[Proposition~2.5]{HW}) it is concluded that 
$\cF$ induces a holomorphic map $X\,\longrightarrow\, \mathcal C_k(X)$, where $\mathcal C_k(X)$ is the Barlet space 
of (compact) cycles of $X$ of dimension $k$. The map is defined by associating to $x$ the cycle $|G_x|\cdot F_x$ 
where $F_x$ is the leaf of $\cF$ though $x$ and $G_x$ is the holonomy group of $F_x$, i.e., the image of 
$\pi_1(F_x)\,\longrightarrow\, \mathrm{Diff}(S,x)$ which is finite by what we explained above. We let $C_{\cF}$ be 
the image of the map $X\,\longrightarrow\,
\mathcal C_k(X)$. We now consider the product map \[f\,:\,X\,\longrightarrow\, C_{\cF}\times C_{\cG}.\]

We claim that $df_x$ is an isomorphism for $x$ general. Indeed, let $x\in X$ be arbitrary, 
let $S$ be a local transversal for $\cF$ at $x$ and let $G$ be the holonomy of $F_x$. The 
local description of the foliations provided by Reeb stability theorem ensures that one can 
find a saturated neighborhood $U_x$ of $x$ and a proper holomorphic map $\pi\,:\,U_x 
\,\longrightarrow\, S/G$ onto the $\ell$-dimensional orbifold $S/G$ such that $\cF 
\,=\,\ker(d\pi)$ holds generically over $S/G$. More precisely, there exists a Zariski open 
subset of $S/G$ over which $\pi$ is smooth so that its differential induces an isomorphism 
in restriction to $\cG$. Now, in restriction to $U_x$, the map $X\,\longrightarrow\, \mathcal C_k(X)$ is 
nothing but the composition $U_x\,\longrightarrow\, S/G\,\longrightarrow\, \mathcal C_k(X)$ where the first arrow is $\pi_x$ 
and the second one is obtained by definition of the Barlet space. We also see in this way 
that $C_\cF$ is a compact $\ell$-dimensional orbifold. Performing the same construction for 
the foliation $\cG$, we see that $f$ is \'etale over a generic point of $C_\cF\times C_\cG$.

As a consequence of the previous paragraph, $f(X)$ is compact of dimension $n$, hence $f$ is surjective. 
Since $\cF$ and $\cG$ are transverse with smooth compact leaves, a leaf of $\cF$ can intersect a leaf of $\cG$ only finitely many times, hence $f$ is finite. 
If the degree of $f$ were greater than one, we could find two points $x,x'\in U$ such that $f(x)=f(x')$. That is, the leaves $F_x$ and $G_x$ would intersect again at the point $x'$. This is a contradiction with the fact that over $U$ via $\phi$, $F_x$ and $G_x$ correspond to $Y\times \{z\}$ and $Z\times \{y\}$ where $\phi(x)=(y,z)$. In particular, their only intersection point on $Y\times Z$ is the point $(y,z)$. 
In conclusion, $f$ is finite and generically $1:1$. Since the source and target of $f$ are normal, Zariski's main theorem implies that $f$ is isomorphic.

It follows from the construction that $\cF$ is isomorphic to $\mathrm{pr}_1^*T_{C_{\cG}}$ (and similarly for $\cG$) via $f$. Moreover, we have seen that a (general) leaf of $\cF$, which isomorphic to a copy of $C_{\cG}$ via $f$, is bimeromorphic to a copy of $Y$ and that map is actually isomorphic in codimension one. After iterating the construction, one will split $X\simeq \prod_{j\in J} X_j$ where $X_j$ is bimeromorphic and isomorphic in codimension one to a possibly singular ICY or IHS K\"ahler variety.

It is then straightforward to check that this implies that $X_j$ itself is an ICY or IHS manifold in the sense of Definition~\ref{def ICYIHS}. The theorem is proved. 
\end{proof}

Using Theorem~\ref{polystability}, the results in \cite{GGK} and 
\cite{BGL} yield the following
 
\begin{corollary}
\label{fund group}
Let $X$ be a compact Fujiki manifold such that $c_1(X)=0\,\in\, {\rm H}^2(X,\, \R)$ and $\widetilde q(X)=0$. Assume either that $X$ is Moishezon or that $\dim X\,\le \,4$. 

Then $\pi_1(X)$ does not admit any finite-dimensional representation with infinite image (over any field). Moreover, for each $k\in \mathbb N$, $\pi_1(X)$ admits only finitely many $k$-dimensional complex representations up to conjugation. 
\end{corollary}

\begin{proof}
In view of the proof of Theorem~\ref{polystability}, there exists a bimeromorphic map $\phi\,:\,X\,\dashrightarrow\, X_{\rm 
min}$ where $X_{\rm min}$ is a K\"ahler variety with klt singularities, zero first Chern class, and zero augmented 
irregularity. The last property comes from the fact that $\phi$ is isomorphic in codimension one, hence there is a 
one-to-one correspondence between finite \'etale covers of $X$ and finite quasi-\'etale covers of $X_{\rm min}$.

If $Z$ is a desingularization of the graph of $\phi$, it follows from Takayama's result \cite{Ta} that the maps $p\,:\,Z\,\longrightarrow\,
X$ and $q\,:\,Z\,\longrightarrow\, X_{\rm min}$ induce isomorphic maps at the level of fundamental groups
\begin{equation}
\label{pi1 2}
p_*\,:\,\pi_1(Z) \,\overset{\simeq}{\longrightarrow}\, \pi_1(X) \quad \mbox{and} \quad q_*\,:\,\pi_1(Z)
\,\overset{\simeq}{\longrightarrow} \, \pi_1(X_{\rm min})
\end{equation}

The corollary is an now an application of \cite[Theorem~I]{GGK} in the projective case and the combination of Theorem~A and Corollary~3.10 in \cite{BGL} in the K\"ahler case.
\end{proof}

\begin{corollary}\label{c2 zero}
Let $X$ be a compact Fujiki manifold such that $0\,=\, c_1(X)\,\in\, {\rm H}^2(X,\, \R)$
and $0\,=\, c_2(X)\,\in\,
{\rm H}^4(X,\,\R)$. Assume either that $X$ is Moishezon or that $\dim X\,\le \,4$. Then $X$ admits a finite \'etale cover $T
\,\longrightarrow\, X$ where $T$ is complex torus. 
\end{corollary}

\begin{proof}
We use the notation from Theorem \ref{polystability}. Consider the vector bundle 
$$p^*TX\,=\,\bigoplus_{i=1}^r p^*F_i,$$ where each $p^*F_i$ has vanishing first Chern class. By the arguments in 
Theorem \ref{polystability}, each $p^*F_i$ is stable with respect $q^*\alpha$. Since 
stability is an open condition, $p^*F_i$ is also stable with respect to a K\"ahler class 
$\theta$ on $X'$. As the first Chern class of $p^*F_i$ vanishes, the Bogomolov--Gieseker's 
inequality says that \[c_2(p^*F_i) \cdot \theta^{n-2}\,\ge\, 0.\] As $c_2(p^*TX)\,=\,\sum 
c_2(p^*F_i)\,=\,0$, we conclude that $c_2(p^*F_i) \cdot \theta^{n-2} \,=\, 0$ for every $i$. By 
Simpson's correspondence, $p^*F_i$ is hermitian flat, and hence $p^*TX$
is hermitian flat. This implies that 
$TX\big\vert_U$ is unitary flat as well, and therefore using fact that the complement
$X\setminus U\, \subset\, X$ is of codimension at least two it follows that $TX$ too is
unitary flat. By \cite[Corollary~1.6]{Dem92}, $X$ is K\"ahler and it follows that $X$
admits a finite \'etale cover by a torus. 
\end{proof}

\begin{remark}
The proof of the corollary above shows that one can weaken the assumption $0\,=\, c_2(X)\,\in\, {\rm H}^4(X,\,\R)$ and
replace it by the existence of a modification $f\,:\,Y\,\longrightarrow\, X$ such that $Y$ admits a K\"ahler form $\omega$
satisfying $c_2(X) \cdot f_*[\omega^{n-2}]\,=\,0$.
 \end{remark}

\begin{remark}
A compact complex manifold in Fujiki class $\mathcal C$ bearing a holomorphic affine connection has vanishing Chern classes, see \cite[Theorem 4 
on p. 192--193]{At}. Therefore, Corollary~\ref{c2 zero} implies that any
compact complex Moishezon manifold admitting a holomorphic affine connection also admits an
\'etale covering by an abelian variety, and is actually projective.
\end{remark}

\subsection{\'Etale triviality of the Albanese map}

We now prove the \'etale triviality of the Albanese map for complex manifolds in Fujiki class $\mathcal C$ with numerically trivial 
canonical bundle, following the strategy of \cite[Lemma 6.1]{Fu} (see also \cite[Theorem~4.1]{CGGN}). 

\begin{theorem}\label{splittorus1}
Let $X$ be a compact complex manifold in Fujiki class $\mathcal C$ such that $c_1(X)\,=\,0\,\in\, {\rm H}^2(X,\, \R)$. Then
after taking some finite \'etale cover, there is a decomposition 
$$X\, =\,T \times Z,$$ 
where $T$ is a torus, and $\widetilde q(Z)=0$.
\end{theorem}

\begin{proof}
Since the canonical bundle $K_X$ is torsion \cite[Theorem 1.5]{To}, we can assume, by replacing $X$ with a finite unramified
covering of it, that $X$ is a compact complex manifold in the Fujiki class 
such that its canonical bundle $K_X$ is holomorphically trivial.

The key point is to prove that the canonical pairing on $X$
$$H^0 (X, \,T_X)\, \times H^0 (X, \,\Omega_X)\,\longrightarrow \,H^0 (X,\, \mathcal{O}_X) \,\simeq\, \mathbb C$$
given by the natural contraction is a perfect pairing (i.e., it is non-degenerate).
 We first prove that $H^0 (X, \,T_X)$ and $H^0 (X, \,\Omega_X)$ have the same complex dimension. This follows from Serre 
duality, Hodge symmetry and triviality of $K_X$:
$$h^{0}(T_X)\,=\,h^n(K_X\otimes \Omega_X)\,=\,h^n(\Omega_X)\,=\,h^1(K_X)\,=\,h^1(\mathcal O_X)\,=\,h^0(\Omega_X).$$

Let $\text{Aut}_0(X)$ be the connected component containing the identity element of the
group of holomorphic automorphisms of $X$. Its Lie algebra is $H^0 (X, \,T_X)$. Let
${\rm alb} \,:\, X \,\longrightarrow\, A(X)$ be the Albanese map. This map is equivariant with respect to the Jacobi homomorphism $$\rho \,: 
\, \text{Aut}_0(X)\,\longrightarrow\, T(X)$$ with $T(X)$ being the compact complex torus
given by the connected component of the group of holomorphic automorphisms of the Albanese manifold $A(X)$.
The differential of $\rho$ is a homomorphism of Lie algebras $$d \rho \,:\, H^0 (X, \,T_X)\, \longrightarrow\, H^0 (X, \,\Omega_X)^*,$$
The kernel of $d \rho$ is the Lie subalgebra ${\mathfrak l}\, \subset\, H^0 (X, \,T_X)$ consisting of holomorphic vector
fields on $X$ that are tangent to the fibers of the map $\rm{alb}$; they are characterized as follows (see \cite[Proposition 6.7]{Fu}):
$${\mathfrak l}\,= \,\{ V \,\in\, H^0 (X,\,T_X)\,\,\mid \,\, \theta (V)\,=\,0\ \, \forall\ \, \theta \,\in\, H^0 (X, \,\Omega_X)\}.$$
If $\mathfrak{l} \,\neq\, \{0 \}$, then $X$ is bimeromorphic to a unirational manifold
\cite[Proposition 5.10]{Fu}. In this case $K_X$ does not admit nontrivial holomorphic sections
(also \cite[Corollary 5.11]{Fu}). Since $K_X$ is holomorphically trivial, we conclude that $\mathfrak{l}\, =\, \{0 \}$.
Consequently $d \rho$ is injective.

Since $\dim H^0 (X, \,T_X)\,=\, \dim H^0 (X, \,\Omega_X)$, we conclude that the injective homomorphism $d \rho$ is an isomorphism.
As $T(X)$ is connected, and $d\rho$ is surjective, it follows that $\rho$ is surjective as well.

Moreover, the kernel of the Jacobi homomorphism $\rho$ is known to be a linear algebraic group \cite[Corollary 5.8]{Fu}. In particular, 
$\mathrm{ker}(\rho)$ has only finitely many connected components. Since the Lie algebra of $\mathrm{ker}(\rho)$
is trivial, it is a finite group. Consequently, 
$\text{Aut}_0(X)$ is a compact complex Lie group, and $\rho \,:\, \text{Aut}_0(X)\,\longrightarrow\, T(X)$ is a finite
covering map. This map factors though a complex Lie group homomorphism $p:\text{Aut}_0(X)\,\longrightarrow\, A(X)$ satisfying 
\begin{equation}
\label{prop p}
\mathrm{alb}(\phi(x))\,=\,\mathrm{alb}(x)+p(\phi), \quad \forall\ \, x\,\in \,X.
\end{equation} 

It is now easy to conclude that $\mathrm{alb}$ is an \'etale trivial fiber bundle onto $A(X)$ (see for example \cite[Lemma 6.1]{Fu} or 
\cite[Theorem~4.1]{CGGN}). We recall the argument here for the reader's convenience. Consider the finite \'etale base change 
$X\times_{A(X)}\text{Aut}_0(X)\,\longrightarrow\,\text{Aut}_0(X)$ by $p$ and let $Z\,=\,\mathrm{alb}^{-1}(0)$. By \eqref{prop p}, the map
\begin{align*}
Z\times \text{Aut}_0(X) &\,\longrightarrow \,X\times_{A(X)}\text{Aut}_0(X)\\
(z,\phi) &\,\longmapsto\, (\phi(z),\,\phi)
\end{align*}
is well-defined, isomorphic and commutes with the projections to the factor $\text{Aut}_0(X)$. This implies our claim on 
$\mathrm{alb}$, and one can check the connectedness of $F$ by observing that the finite factor in the Stein factorization of 
$\mathrm{alb}$ is \'etale and has a section by the universal property of the Albanese map.

Finally, the Albanese map of $X\times_{A(X)}\text{Aut}_0(X)\,\simeq\, Z\times \text{Aut}_0(X)$ coincides with the projection to the 
second factor. In particular, we have $h^{1,0 }( Z)\,=\,0$. Now, if $Z'\,\longrightarrow\, Z$ is an \'etale cover such that 
$h^{1,0}(Z')\,\neq\, 0$, we repeat the construction on $Z'$ to split off another torus factor after a further \'etale cover of $Z'$, and 
the process stops after finitely many steps.
\end{proof}

Let $X$ be a compact manifold in Fujiki class $\mathcal C$ with trivial canonical bundle. It is said to admit an {\it algebraic 
approximation} if there is a small deformation $\mathcal{X} \,\longrightarrow\, \Delta$ of $X$, with $X\,=\,X_0$ being the central fiber,
and a sequence of points $t_i \,\in\,\Delta$, $i\, \in\, \mathbb N$, converging to $0$, such that all the fibers
$\mathcal{X}_{t_i}$ are Moishezon.

\begin{lemma}\label{density}
Let $X$ be a compact complex manifold in Fujiki class $\mathcal C$ such that $c_1(X)\,=\,0\,\in\, {\rm H}^2(X,\, \R)$ and 
$c_2(X)\,=\,0\,\in\, {\rm H}^4(X,\,\R)$. Assume that $X$ admits an algebraic approximation. Then $X$ admits a finite unramified
covering by a compact complex torus.
\end{lemma}	 
 
\begin{proof} 
By Theorem \ref{splittorus1}, up to a finite \'etale cover, we have a decomposition $X \,=\,T \times Z$ where $h^{1,0 }( Z)\,=\,0$ and $T$ is a
compact complex torus. The given condition that $c_2(X)\,=\,0$ implies that $c_2(Z)\,=\,0$.
Let $\mathcal{X}\,\longrightarrow\, \Delta$ be an algebraic approximation of $X$ and let $t_0\in \Delta$ such that $X_{t_0}$ is Moishezon. Since $c_2(X)\, =\,0$, it follows that $c_2 (X_{t_0})\,=\,0$.

Corollary \ref{c2 zero} applies to the Moishezon manifolds $X_{t_0}$ to show that a finite \'etale cover of
$X_{t_0}$ is an abelian variety. Since $\mathcal X\,\longrightarrow\, \Delta$ is $C^{\infty}$-trivial, one can extend the latter cover to 
a finite \'etale cover of the family $\mathcal{X}\,\longrightarrow\, \Delta$. This enables us to assume
that $X_{t_0}$ is torus. Since $t\,\longmapsto\, b_1(X_t)$ is constant and both $X$ and $X_{t_0}$ are Fujiki, it follows that
$h^{1,0}(X)\,=\, h^{1,0} (X_{t_0})\,=\,\dim X$. Since $X\,=\,T\times Z$ and $h^{1,0}(Z)\,=\,0$, we get 
$\dim T\,=\,\dim X$ so that $X\,=\,T$ is a compact complex torus.
\end{proof}

\begin{corollary}
\label{c_2=0, Fujiki}
Let $X$ be a compact complex manifold in Fujiki class $\mathcal C$ such that $c_1(X)\,=\,0\,\in\, {\rm H}^2(X,\, \R)$ and 
$c_2(X)\,=\,0\,\in\, {\rm H}^4(X,\,\R)$. Assume that $X$ has algebraic dimension 
$a(X)\,=\, \dim X -1$. Then $X$ admits a finite \'etale cover $T\,\longrightarrow\, X$ where $T$ is a compact complex torus.
\end{corollary}

\begin{proof}
This is a direct consequence of Lemma~\ref{density} and of the approximation result \cite[Corollary 1.4]{Lin} stating that
a Fujiki class $\mathcal C$ manifold $X$ of algebraic dimension $a(X)\,=\,\dim X-1$ admits an algebraic approximation.
\end{proof}

\begin{remark}\label{unobstructed}
While compact K\"ahler manifolds with numerically trivial canonical bundle are known to admit algebraic approximation \cite{Ca}, it is 
not known whether all Fujiki class $\mathcal C$ manifolds with numerically trivial canonical bundle admit algebraic approximations. 
Despite the unobstructedness of the Kuranishi space (\cite[Theorem 1.2]{Po} and \cite[Theorem 3.3]{ACRT}), a major stumbling block
in adapting the arguments in \cite{Ca} is the fact that Fujiki class $\mathcal C$ manifolds are 
not stable under deformations, already in dimension three \cite{Cam}.
\end{remark}

\section{Geometric structures on Moishezon manifolds}\label{volse}

The aim of this section is to prove the following theorem on holomorphic geometric structures on Fujiki manifolds with numerically 
trivial canonical bundle that are either Moishezon or of dimension no greater than four.

\begin{theorem}
\label{moishezon}
Let $X$ be a compact Fujiki manifold such that $c_1(X)\, \in\, {\rm H}^2
(X,\, \mathbb R)$ vanishes. Assume either that $X$ is Moishezon or that $\dim X\le 4$. 

\begin{enumerate}[label=(\roman*)]
\item There exists a Zariski open set $U \,\subset\, X$, whose complement has complex codimension
at least two, such that any holomorphic geometric structure $\phi$ of affine type
on $X$ is locally homogeneous on $U$. 
\item If there exists a rigid holomorphic geometric structure $\phi$ of affine type on $X$, then there exists a finite \'etale cover $T
\,\longrightarrow\, X$, where $T$ is a complex torus. The pull-back of $\phi$ on $T$ is translation invariant.
\end{enumerate}
\end{theorem}

\begin{remark}
\label{rem fund group}
If Conjecture~\ref{conj structure} were to hold, then the proof of Theorem~\ref{moishezon} below would apply to show that its conclusions are valid for any compact Fujiki manifold with trivial first Chern class. A weaker but unconditional result will be proven in that direction (see Corollary~\ref{thmFujiki}). 
\end{remark}

\begin{proof}
Let $U$ be the Zariski open set defined in the statement of Theorem \ref{polystability}, then the first statement follows from 
Lemma~\ref{loc hom} and the Bochner principle proved in Theorem~\ref{polystability}.

Next assume that $\phi$ is rigid. From Theorem~\ref{polystability} we know that $X$ satisfy the Decomposition 
conjecture. Hence one can find a finite \'etale cover $X'\,=\,T\times Y\,\longrightarrow\, X$ where $Y$ 
has zero augmented irregularity and is either Moishezon or has dimension at most four. Since the tangent 
bundle of $T$ is trivial, $\phi$ induces a rigid holomorphic geometric structure $\phi_Y$ of affine type on 
$Y$ as explained in the proof of \cite[Theorem~2]{D2}. By the first part, $\phi_Y$ is locally homogeneous on 
a Zariski open set of $Y$.

We claim that there exists $g\,:\,Y'\,\longrightarrow\, Y$ a finite \'etale cover such that any complex linear 
representation of $\rho\,:\,\pi_1(Y')\,\longrightarrow\, \mathrm{GL}(n,\mathbb C)$ is trivial, where $n\,=\,\dim Y$. 
Indeed, by Corollary~\ref{fund group} there are only finitely many classes $\{[\rho_i]\,\big\vert\, i\,\in \,I\}$ of such 
representations up to conjugation and each $\rho_i$ has finite image. The sought cover $g$ then corresponds to the 
subgroup of finite index of $\pi_1(Y)$ obtained as the intersection $\cap_I \mathrm{ker}(\rho_i)$. Up to replacing 
$Y$ with $Y'$, we can therefore assume that $\pi_1(Y)$ has no non-trivial complex representation of dimension $n$. Moreover, up to applying Theorem~\ref{polystability} again, one can assume that $Y=\prod_{i\in I} Y_i$ is a product of ICY and IHS manifolds. 

In view of Remark~\ref{rep pi1}, non-trivial local Killing fields for the rigid structure $\phi_Y$ (which 
exist since $\phi_Y$ is locally homogeneous on a non-empty Zariski open set of $Y$) can be extended globally 
to $Y$. In particular, we find that $H^0(Y, \,T_Y)\,\neq\, 0$. Now, if we set $n_i
\,=\,\dim Y_i$, then the 
identity $K_{Y_i}\,\simeq\, \mathcal O_{Y_i}$ implies that $H^0(Y_i,\,T_{Y_i})\,\simeq\, H^0(Y_i,\, 
\Omega^{n_i-1}_{Y_i})$, and the latter space is zero by the very definition of ICY and IHS
manifolds. In particular, we must have $H^0(Y,\, T_Y)\,=\, 0$, which provides the expected contradiction.
\end{proof}

\begin{remark} \label{rational}
In Theorem \ref{moishezon} the condition on triviality of $c_1(X)$ is essential. To illustrate this, consider 
the complex projective line equipped with two holomorphic nonzero global vector fields $v_1$ and $v_2$ such that 
the two divisors ${\rm div}(v_1)$ and ${\rm div}(v_2)$ are actually disjoint. The geometric structure $\phi$ on 
the complex projective line given by $v_1,\, v_2$ is a holomorphic rigid geometric structure of affine type. Note 
that $\phi$ is not locally homogeneous on any nonempty open subset of the projective line. Indeed, the quotient 
$v_1/v_2$, which is a non-constant meromorphic function, is actually a scalar invariant of $\phi$; it cannot be 
constant on any nonempty open subset of the projective line.
\end{remark}

The following is derived using Theorem \ref{moishezon}.

\begin{corollary}\label{torus bundle}
Let $\mathcal T$ be a compact complex torus.
Let $\pi \,:\, X \,\longrightarrow\, B$ be a
holomorphic principal $\mathcal T$--bundle over a compact Moishezon manifold $B$
with numerically trivial canonical bundle $K_B$. Then there exists a $\mathcal T$-invariant Zariski open set
$U \,\subset\, X$, whose complement has complex codimension at least two, such
that any holomorphic geometric structure $\phi$ of affine type
on $X$ is locally homogeneous on $U$. If such a $\phi$ is moreover rigid, then
the fundamental group of $X$ must be infinite.
\end{corollary}

\begin{proof} Theorem 1.5 of \cite{To} shows that there exists a positive integer $\ell$ such
that $K_B^{\otimes \ell}$ is holomorphically trivial. We consider the index one cover $B'\,\longrightarrow\, B$; this is a finite unramified cover of $B$ such that $K_{B'}$ is trivial. 
Up to replacing $X$ with $X\times_B B'$, we may assume that $K_B$ is trivial. 

We claim that $K_X$ is trivial. Indeed, since $X$ is a principal $\mathcal T$-bundle, any trivialization $\Omega_T$ of $K_T$ glues to a 
trivialization $\Omega_{X/B}$ of $K_{X/B}$. If $\Omega_B$ is a trivialization of $K_B$, then $\Omega_{X/B}\wedge \Omega_B$ yields a 
trivialization of $K_X$.

Consider the Zariski open dense subset $U \,\subset\, B$ defined in the statement of Theorem~\ref{polystability}. Also
consider the Zariski open set $\pi^{-1} (U) \,\subset\, X$.
We will prove that any holomorphic geometric structure of affine type on $X$ is locally homogeneous on $\pi^{-1} (U)$.

To prove this by contradiction, assume that there exists a holomorphic geometric structure of 
affine type on $X$ which is not locally homogeneous on $\pi^{-1} (U)$. By Lemma~\ref{loc hom}, there exist a point 
 $x \,\in\, \pi^{-1}(U)$, integers $a,\,b \,\geq\, 0$ and a nontrivial holomorphic section
\begin{equation}\label{f4}
\eta\, \in\, {\rm H}^0(X,\, (TX)^{\otimes a} \otimes (T^*X)^{\otimes b})\setminus\{0\}
\end{equation}
such that $\eta(x)\,=\, 0$; note that this condition implies that $a+b\, >\, 0$.

Since $K_X$ is trivial, contraction by $TX$ of a fixed nonzero
holomorphic section of $K_X$ produces a holomorphic isomorphism
$ TX \simeq \Lambda ^{n-1} (T^*X)$ where $n\,=\, \dim_{\mathbb C} X$. Therefore, the section $\eta$ in \eqref{f4}
produces
$$s\, \in\, {\rm H}^0(X,\, (T^*X)^{\otimes r}) \setminus\{0\},
$$
where $r\,=\,(n-1)a+b$, such that $s(x)\,=\, 0$.

Then the proof of \cite[Theorem 1.2]{BD2} shows that $s$ is the pull-back of a 
holomorphic section $t \, \in\, {\rm H}^0(B,\, (T^*B)^{\otimes r})$; the
condition that $s(x)\,=\, 0$ implies that 
$t(\pi(x))\,=\, 0$. Since $\pi(x) \in U$, this is in contradiction with the Bochner principle proved in Theorem \ref{polystability}. 

Finally, assume that $\phi$ is rigid. As in the proof of Theorem 
\ref{moishezon}, since $\phi$ is locally homogeneous on an open dense 
subset of $X$, Proposition \ref{loc hom2} 
shows that the fundamental group of $X$ is infinite. 
\end{proof} 

As a consequence of Theorem \ref{moishezon} we have the following:

\begin{corollary}\label{thmFujiki}
Let $X$ be a compact complex manifold in Fujiki class $\mathcal C$ such that $c_1(X)\, \in\, {\rm H}^2
(X,\, \mathbb R)$ vanishes. If $X$ bears a holomorphic rigid geometric structure of affine type, then
the fundamental group of $X$ must be infinite.
\end{corollary}

Note that the conclusion of Corollary~\ref{thmFujiki} is weaker than that of Theorem~\ref{moishezon} (ii), but it applies to any Fujiki manifold (see Remark~\ref{rem fund group}).
 
\begin{proof}
This follows from Theorem \ref{moishezon} if $X$ is a Moishezon manifold. Assume that $X$ is not a Moishezon manifold.
Since the algebraic dimension of $X$ is not maximal,
Theorem 4.2 of \cite{BD2} implies that the fundamental group of $X$ is infinite.
\end{proof}

Another consequence of Theorem \ref{moishezon} is the following:

\begin{corollary}\label{Kill rigid}
Let $X$ be a compact complex manifold with $c_1(X)\,=\,0 \,\in\, H^2(X,\,\R)$. Then any holomorphic
rigid geometric structure of affine type on $X$ admits a non-trivial Lie algebra of (local) Killing vector fields.
\end{corollary}

\begin{proof}
To prove this statement by contradiction, assume that the Lie algebra of Killing vector fields for $\phi$ is trivial. Then \cite[Theorem 
2.1]{D1} (see also \cite[Theorem 3]{D2}) implies that the fibers of the algebraic reduction of $X$ have 
dimension zero, meaning $X$ is Moishezon. But Theorem \ref{moishezon} says that $\phi$ is locally homogeneous on an 
open dense subset of $X$. Consequently, the Lie algebra of Killing vector fields is transitive on an open
dense subset in $X$, in particular, the Lie algebra of Killing vector fields is non-trivial. In view of this contradiction the
proof is complete.
\end{proof}

\section{Automorphism group and fibrations by complex tori}\label{section main results}

This section is devoted to generalizing Corollary~\ref{thmFujiki} to other classes of 
manifolds, namely compact complex manifolds with algebraic dimension at most one and trivial 
canonical bundle and compact complex threefolds with trivial 
canonical bundle (see Corollary~\ref{thm part cases}). The combination of Corollary~\ref{thmFujiki}, and Corollary~\ref{thm part cases} yields
Theorem \ref{main}.

The following statement is the main result of this section, from which we will easily derive Corollary~\ref{thm part cases}.

\begin{theorem}
\label{main lemma}
Let $X$ be a compact simply connected complex manifold with trivial canonical
bundle $K_X$. Assume that $X$ bears a holomorphic rigid geometric structure $\phi$ of affine type. Then the following statements
hold:
\begin{enumerate}
\item[(i)]\, There exists a holomorphic submersion $\pi \,:\, X\,\longrightarrow\, B$ to a simply connected Moishezon manifold $B$ with
globally generated canonical bundle $K_B$ such that the fibers of $\pi$ are complex tori.

\item[(ii)]\, The fibration $\pi$ is not isotrivial. Equivalently, $K_B$ is not trivial.

\item[(iii)]\, There exists a maximal connected abelian subgroup $A$ of the automorphism group $\rm{Aut}(X,\, \phi)$ whose orbits 
coincide with the fibers of $\pi$. Moreover, $A$ is noncompact and its (real) maximal compact subgroup $K$ acts freely and transitively 
on the fibers of $\pi$ (hence $X$ is a $C^\infty$ principal $K$--bundle over $B$).
\end{enumerate}
\end{theorem}

\begin{proof}
Take a simply connected compact complex manifold $X$ endowed with a rigid holomorphic geometric structure $\phi$ of affine type.
As done in the proof of Theorem \ref{act alg red}, consider the Lie subalgebra of 
${\rm H}^0(X,\, TX)$ corresponding to the subgroup ${\rm Aut}_0(X,\, \phi)\, \subset\,{\rm Aut}_0(X)$, and fix a basis
$$\{X_1,\, \ldots,\, X_{k}\} \,\subset\, \text{Lie}({\rm Aut}_0(X,\, \phi))\, \subset\,{\rm H}^0(X,\, TX)$$
of this subalgebra. Note that the latter subalgebra is non-trivial (i.e., $k\ge 1$) by Corollary~\ref{Kill rigid} coupled with Theorem~\ref{extend}. 

Let $\phi' \,=\,(\phi,\, X_1,\, \ldots ,\, X_k)$ be the rigid geometric structure on $X$;
its automorphism group is denoted by ${\rm Aut}(X,\, \phi')$. From Theorem \ref{act alg red}
we know that the maximal connected subgroup $A\,:=\,{\rm Aut}_0(X,\, \phi')$ of
${\rm Aut}(X,\, \phi')$ is abelian. 
Since $A$ preserves a smooth measure on $X$ (see Lemma \ref{vol}) its orbits 
are compact and they coincide with the orbits of the maximal compact (real) subgroup
\begin{equation}\label{f6}
K\, \subset\, A
\end{equation}
(see \cite[Section 3.7]{Gr} and \cite[Section 3.5.4]{DG}). It should be mentioned that the proof in \cite{Gr,DG} first shows that 
the stabilizer $A(x)$ of any point $x \,\in\, X$ for the action of $A$ on $X$ is an algebraic group, and hence the stabilizer has only 
finitely many connected components; then their proof uses the main result of \cite{Mon} which asserts that for a homogeneous space 
$A/A(x)$ with a finite $A$--invariant measure, if $A(x)$ has only finitely many connected components, then $A/A(x)$ is compact and 
furthermore the action of the maximal compact (real) subgroup $K$ of $A$ on $A/A(x)$ is transitive.

\medskip

{\bf Step 1. All $A$--orbits in $X$ have the same dimension. }

Arguing by contradiction, assume that there exists an $A$--orbit $$\textbf{O}\,\, \subset\,\, X$$ whose 
dimension is strictly less than the maximum of the dimensions of the $A$-orbits in $X$. Since
$A$-orbits and $K$-orbits coincide (see \eqref{f6}), from the above
property of $\textbf{O}$ it follows immediately that the stabilizer
$K_o\, \subset\, K$ of a point $o\,\in\, 
\textbf{O}$ for the action of $K$ is a real Lie subgroup of positive dimension. The maximal connected 
subgroup $K^0_o\, \subset\, K_o$ is a compact connected abelian group, and hence we have
\begin{equation}\label{f2}
K^0_o \,=\, (S^1)^l
\end{equation}
for some positive integer $l$. The action of $K^0_o$ linearizes locally
(on some neighborhood of $o$ in $X$), meaning it is locally 
isomorphic to the linear action of $K^0_o$ on $T_oX$ given by the differential
of the action of $K^0_o$ on $X$. The group $K$ being abelian, the 
action of $K^0_o$ on the tangent space of the orbit $T_o\textbf{O}$ is actually trivial;
indeed, the action of $K^0_o$ on $T_o\textbf{O}$ is conjugate 
to the restriction, to $K^0_o$, of the adjoint action of $K$ on the quotient
${\rm Lie}(K)/{\rm Lie}(K^0_o)$ of Lie algebras. Moreover, $K^0_o$ preserves the 
orthogonal complement
\begin{equation}\label{f7}
V_o \,=\,(T_o\textbf{O})^{\perp}\, \subset\, T_oX
\end{equation}
with respect to any $K^0_o$--invariant Hermitian form on $T_oX$. Fix a
$K^0_o$--invariant Hermitian form on $T_oX$.

Denote by $A_o$ the stabilizer of $o\, \in\, \textbf{O}\, \subset\, X$ for the action of $A$ on $X$. It is a 
complex Lie subgroup of $A$; in fact $A_o$ is identified with a complex algebraic subgroup of $D^{r+s}$ for some 
integer $s$ \cite{Gr, DG}. Denote by $C$ the smallest connected complex Lie subgroup of $A$ containing $K^0_o$; 
the Lie algebra of $C$ is the image of ${\rm Lie}(K^0_o)\oplus \sqrt{-1}{\rm Lie}(K^0_o)$ in ${\rm Lie}(A)$. Since 
the $K^0_o$--action on $T_oX$ is linearizable, the $C$-action on $T_oX$ is linearizable as well (in fact, it
is linearizable with 
respect to the same coordinates). More precisely, $C$ is isomorphic to $({\mathbb C}^*)^l$ acting on $V_o$ 
diagonally (see \eqref{f2} and \eqref{f7}); in other words, $V_o$ splits as a direct sum of complex lines
\begin{equation}\label{eed}
V_o\, =\,L_1 \oplus \cdots\oplus L_p,
\end{equation}
such that on each direct summand $L_i$ the group $C$ acts through a character defined by
\begin{equation}\label{f8}
(t_1, \,\ldots,\, t_l) \,\longmapsto\, t_1^{n_{1,i}} \cdot t_2 ^{n_{2,i}} \cdot \cdots
\cdot t_l^{n_{l,i}}
\end{equation}
for every $(t_1, \,\ldots,\, t_l)\,\in \,({\mathbb C}^*)^l$. Note that $C$ acts trivially on $T_o\textbf{O}$, because
$K^0_o$ acts trivially on $T_o\textbf{O}$.

Assume that the character of $C$ in \eqref{f8} corresponding to an eigenline $L_i$ in \eqref{eed} is not trivial, in 
other words, $n_{j,i}\, \not=\, 0$ for some $1\, \leq\, j\, \leq\, l$ (see \eqref{f8}). 
Then $L_i \setminus \{ 0 \}\, \subset\, V_o$ is a $C$--orbit; the origin $0\, \in\, V_o$ is 
an accumulation point of this orbit. From this it follows that for any open neighborhood $U\, 
\subset\,X$ of $o$ on which the $K^0_o$--action is linearizable, there are $C$--orbits 
(and hence $A$--orbits) in $U\setminus (U \bigcap \textbf{O})$ for which $o \,\in 
\,U\bigcap \textbf{O}$ is an accumulation point.

But this contradicts the fact that the orbits of $A$ in $X$ are locally closed (in fact 
they are compact). Therefore, we conclude that $n_{j,i}\,=\, 0$ for all
$1\, \leq\, j\, \leq\, l$ and $1\, \leq\, i\, \leq\, p$ (see \eqref{f8}).

Consequently, the $K^0_o$--linear action on $T_oX$ must be trivial, and hence 
the $K^0_o$--action is trivial on the open neighborhood of $o$ where the action 
is linearized. Using analyticity, this implies that the $K^0_o$--action on $X$ is trivial. This is
a contradiction because $l$ in \eqref{f2} is positive.

Therefore, we conclude that all $A$--orbits in $X$ have the same dimension.

\medskip

{\bf Step 2. The action of $K$ on $X$ is free. }

Take any point $x_0 \,\in\, X$; let $K(x_0)\, \subset\, K$ be the identity component of the stabilizer of $x_0$ for the action of $K$ on $X$. 
Since $K(x_0)$ is compact, its action linearizes on a neighborhood of $x_0$ in $X$. For 
any $k \,\in\, K(x_0)$, the differential $dk(x_0)$ acts trivially on $T_{x_0}(Kx_{0})\, 
=\,T_{x_0}(Ax_0)$; indeed, as mentioned earlier, this action is induced by the restriction, to 
$K(x_0)$, of the adjoint representation of $K$, and this adjoint action is trivial because $K$ 
is abelian. Since $K(x_0)$ is compact, the orthogonal complement of $T_{x_0}(Kx_{0})\, 
=\,T_{x_0}(Ax_0)\, \subset\, T_{x_0} X$, for a $K(x_0)$--invariant Hermitian form on $T_{x_0}X$, is 
actually $K(x_0)$--invariant.

On the other hand, the differential $dk(x_0)$ acts trivially on $T_{x_0}X/T_{x_0}(Ax_0)$ 
because the action of any element of $A$ (in particular of $k$) preserves each $A$--orbit; in
other words, $k$ acts trivially on the space of $A$--orbits.

This implies that the linear action of $K(x_0)$ on $T_{x_0}X$ is trivial. From this
it follows that the action of $K(x_0)$ on $X$ is trivial; indeed, the action of $K(x_0)$ is
linearizable, so on any open neighborhood of $x_0 \,\in\, X$ on which 
the action of $K(x_0)$ is linearizable, the action of $K(x_0)$ is trivial, because the action
of $K(x_0)$ on $T_{x_0}X$ is trivial.
In view of this, since the $A$--action, and therefore the $K$--action, on $X$ is faithful, we now conclude
that the group $K(x_0)$ is trivial. Consequently, the action of $K$ on $X$ is actually free.

\medskip 

{\bf Step 3. Fibers of $X\,\longrightarrow\, X/K$ are tori. }

It now follows that $X$, equipped with the action of $K$, has the structure of a real principal $K$--bundle over 
the smooth real manifold $B\,=\, X/K$. The $K$--orbits are complex submanifolds because they are also $A$--orbits 
and $A$ is a complex Lie group acting holomorphically on $X$. This implies that $B$ is also a complex manifold, 
and the projection
\begin{equation}\label{f9}
\pi\,\,:\,\, X\, \,\longrightarrow\,\, X/K\,\,=\,\, B
\end{equation}
is in fact a holomorphic submersion.

Since $A$ is abelian, and any $A$--orbit $Ax_0$ is identified with $A/(A(x_0))$ with $A(x_0)$ being the stabilizer of 
$x_0$ for the action of $A$ on $X$, the holomorphic tangent bundle $T(Ax_0)$ is holomorphically trivial.
In fact, a holomorphic trivialization of 
$T(Ax_0)$ is obtained by fixing a basis of the abelian Lie algebra ${\rm Lie}(A)/{\rm Lie}(A(x_0))$ (the adjoint 
representation of $A$ is trivial because $A$ is abelian). As every $A$--orbit is compact with its holomorphic tangent bundle 
trivialized by commuting holomorphic vector fields, we conclude that every $A$--orbit in $X$ is a compact complex 
torus \cite{Wa}.

\medskip

{\bf Step 4. $K_B$ is globally generated. }

First, recall that since the fibers of $\pi$ in \eqref{f9} are connected and $X$ is simply
connected, the base $B$ is simply connected as well.

Denote by $m$ the complex dimension of the fibers of $\pi$. Take a point
\begin{equation}\label{f12}
b\,\, \in\,\, B\, .
\end{equation}
Choose a family of holomorphic vector fields on $X$ belonging to the Lie algebra of $A$ $$(X_1,\, \ldots,\, X_m)$$ 
(they are chosen from the fundamental vector fields for the $A$--action) satisfying the condition that their 
restrictions to $\pi^{-1}(b)$ span the tangent bundle $T(\pi^{-1}(b))$, hence they form a basis of the latter). This implies that the locus of all $b'\, 
\in\, B$ such that the restrictions of $(X_1,\, \ldots,\, X_m)$ to $\pi^{-1}(b')$ span $T(\pi^{-1}(b'))$ is an 
open dense subset of $B$ whose complement is a closed complex analytic subspace.

Let
\begin{equation}\label{f10}
\omega\, \in\, {\rm H}^0(X,\, K_X) \setminus\{0\}
\end{equation}
be a trivializing section of $K_X$. Then the holomorphic form
\begin{equation}\label{f11b}
\omega'\, :=\, i_{X_1}\circ i_{X_2}\circ\cdots \circ i_{X_m}\omega
\, \in\, {\rm H}^0(X,\, \Omega^{n-m}_X)\setminus\{0\}\, ,
\end{equation}
where $n\,=\, \dim_{\mathbb C}X$, satisfies the equation
\begin{equation}\label{f11}
\omega'\, \,=\,\,\pi^*\widehat{\omega}
\end{equation}
for some $\widehat{\omega}\, \in\, {\rm H}^0(B,\, K_B)\setminus\{0\}$,
with $\pi$ being the projection in \eqref{f9}. Note that $\widehat{\omega}(b')\, \not=\, 0$
for all $b'\, \in\, B$ such that the restrictions of $(X_1,\, \ldots,\, X_m)$ to $\pi^{-1}(b')$
span $T(\pi^{-1}(b'))$; in particular, $\widehat{\omega}(b)\, \not=\, 0$
where $b$ is the point in \eqref{f12}. Now moving $b$ over $B$ we conclude that
$K_B$ is generated by its global holomorphic sections. 

\medskip

{\bf Step 5. $B$ is Moishezon. }

\medskip
Using the notation of Theorem~\ref{thue}, consider the algebraic reduction $t\,:\, \widetilde X\,\longrightarrow\, V$ and set $\rho:=\pi \circ \psi \,:\, \widetilde X\,\longrightarrow\, B$. We have a diagram as follows. 
\begin{equation*}
\begin{tikzcd}
\widetilde X \ar[dr, "t", swap] \ar[r, "\psi"] \ar[rr, "\rho", bend left] & X \ar[d, dotted] \ar[r,"\pi"] & B \\
&V& 
\end{tikzcd}
\end{equation*}
The maps $\rho$ and $t$ are proper, surjective with connected fibers, and by Theorem~\ref{act alg red}, $\rho$ contracts every fiber of $t$. It is then classical that $\rho$ factors through $t$, that is, there exists a map $\sigma \,:\, V\,\longrightarrow\, B$ such that $\rho=\sigma \circ t$. Clearly, $\sigma$ is surjective; since $V$ is Moishezon, so is $B$ by \cite[Theorem 2]{Mo}. 

Let us briefly recall how $\sigma$ is constructed, for the reader's convenience. Consider the image $Z$ of $\rho \times t\,:
\,\widetilde X \,\longrightarrow\, B\times V$ and the map $g\,:\, Z\,\longrightarrow\, V$ induced by the projection $\mathrm{pr}_V
\,:\,B\times V \,\longrightarrow\, V$. Let $\nu\,:\, Z^\nu\,\longrightarrow\, Z$ be the normalization of $Z$. Next, $g$ is
surjective, proper and by assumption, one has $|g^{-1}(v)|\,=\,1$ for any $v\,\in\, V$. By Zariski main theorem, the map
$g^\nu\,:\,Z^\nu \,\longrightarrow\, V$ is an isomorphism; set $h\,:=\,\nu \circ (g^\nu)^{-1}\,:\,V\,\longrightarrow\, Z$. If one
defines $\sigma\,:\, V\,\longrightarrow\, B$ by $\sigma\,=:\,\mathrm{pr}_B\circ h$, then we have $\rho\,=\,\sigma \circ t$ as
desired, and have the following commutative diagram:
\[
\begin{tikzcd}
\widetilde X \ar[dr, "t", swap] \arrow[r] \ar[rrr, "\rho", bend left] & Z \ar[d, "g", swap] \arrow[hookrightarrow]{r} & B\times V \ar[r, "\mathrm{pr}_B"] & B \\
&V \ar[rru, "\sigma", swap] \ar[u, "h", bend right, swap]& & 
\end{tikzcd}
\]
This completes the proof of $(i)$.

\medskip

{\bf Step 6. $K_B$ is not trivial and $\pi$ is not locally trivial. }

Argue by contradiction and assume that $K_B$ is trivial. We will prove that the group $K$ has a complex (torus) structure such that the 
quotient map $\pi$ in \eqref{f9} makes $X$ a holomorphic principal $K$--bundle over $B$. Since $X$ is simply connected and $B$ is 
Moishezon by the previous step, this will contradict Corollary~\ref{torus bundle}.

Going back to the proof, our assumption implies that the
holomorphic section 
$\widehat{\omega}\, \in\, {\rm H}^0(B,\, K_B)\setminus\{0\}$ constructed in \eqref{f11} does 
not vanish at any point, and hence $\widehat{\omega}$ trivializes
$K_B$ holomorphically. In view of \eqref{f11b} and \eqref{f11},
this implies that the family of holomorphic vector fields
$$\{X_1,\, \ldots,\, X_m\}$$
in \eqref{f11b} satisfies the condition that $\{X_1(x),\, \ldots,\, X_m(x)\}\, \subset\,
T_xX$ are linearly independent for every $x\, \in\, X$. Consequently,
the natural evaluation map
\begin{equation}\label{b1}
\beta\,\,:\,\, \bigoplus_{j=1}^m {\mathbb C}\cdot X_j\,\,\longrightarrow\,\,
T_{X/B}\, ,
\end{equation}
where $T_{X/B}\, \subset\, TX$ is the relative holomorphic tangent bundle
for $\pi$ in \eqref{f9}, is a holomorphic isomorphism. For any $1\, \leq\, j\, \leq\, m$,
let $X'_j\, \in\, {\rm Lie}(A)$ be the element corresponding to the vector field
$X_j$. Let
\begin{equation}\label{b2}
{\mathcal H} \,\, \subset\,\, {\rm Lie}(A)
\end{equation}
be the complex subspace generated by $\{X'_1,\, \ldots,\, X'_m\}$. Let
$$
{\rm Lie}(K)_{\mathbb C}\,\, \subset\, \,{\rm Lie}(A)
$$
be the complex subspace generated by ${\rm Lie}(K)\, \subset\, {\rm Lie}(A)$.

For any $v\, \in\, {\rm Lie}(K)_{\mathbb C}$, there are complex valued functions $f^v_1,\,
\ldots ,\, f^v_m$ on $B$ such that the holomorphic vector field $v'$ on $X$
corresponding to $v$ satisfies the equation
$$
v'(b)\,=\, \sum_{j=1}^m f^v_j(b)\cdot X'_j(b)
$$
for all $b\, \in\, B$. Indeed, this follows immediately from the fact that
$\beta$ in \eqref{b1} is an isomorphism. The functions $f^v_1,\, \ldots,\, f^v_m$ are
evidently holomorphic, and hence they are constants. This implies that
\begin{equation}\label{b3}
{\rm Lie}(K)_{\mathbb C}\, \subset\, {\mathcal H}\, ,
\end{equation}
where ${\mathcal H}$ is the subspace in \eqref{b2}. On the other hand,
$$
\dim_{\mathbb C} {\mathcal H}\,=\, m\, \leq\, \dim_{\mathbb C}{\rm Lie}(K)_{\mathbb C}\, ,
$$
because $\dim_{\mathbb R} {\rm Lie}(K)\,=\, 2m$. This and \eqref{b3} together imply 
that ${\mathcal H}\,=\, {\rm Lie}(K)_{\mathbb C}$. But this implies
that ${\rm Lie}(K)_{\mathbb C}\,=\, {\rm Lie}(K)$, because
$2\cdot\dim_{\mathbb C} {\mathcal H}\,=\, 2\cdot\dim_{\mathbb C} {\rm Lie}(K)_{\mathbb C}
\,=\, \dim_{\mathbb R} {\rm Lie}(K)$.

Since ${\rm Lie}(K) \,=\, {\rm Lie}(K)_{\mathbb C}$, we conclude that $K$ is a 
complex Lie subgroup of $A$, in particular, $K$ is a compact complex torus.
Therefore, the projection $\pi$ in \eqref{f9} makes 
$X$ the total space of a {\it holomorphic principal} $K$--bundle over $B$. This is the sought contradiction (in view of Corollary~\ref{torus bundle}) and $K_B$ is not trivial. 

Let us now show that $\pi$ is not locally trivial arguing by contradiction. We denote by $T$ the fiber, which is a complex torus. Since $\pi$ is locally trivial, one can consider the monodromy operator $\tau_\gamma\in \mathrm{Aut}(T)$ of the fibration associated to a given loop in $B$. Since $B$ is simply connected, we actually have $\tau_\gamma \in \mathrm{Aut}_0(T)\simeq T$. Since any trivialization $\sigma_T$ of $K_T$ is preserved by $\mathrm{Aut}_0(T)$, parallel transport allows us to extend $\sigma_T$ to a trivialization $\sigma_{X/B}$ of $K_{X/B}$. As $K_X$ is trivial, this implies that $K_B$ is trivial.
This is in contradiction with what we have just proved.

The combination of our arguments in the current Step 6 show that, in particular, the local triviality of $\pi$ is equivalent to the 
triviality of $K_B$. Note that this was also proved by Catanese, Oguiso and Peternell, see \cite[Theorem 3.1]{COP}. This completes the 
proof of item $(ii)$.

\medskip

Finally, one can now easily see that $A$ is non-compact. Indeed, arguing by contradiction, $A$ coincides with its maximal compact 
subgroup $K$; so $A$ is a compact complex torus. We have seen that $K$ acts freely and transitively on the fibers of $\pi$. Therefore, 
$\pi$ defines an principal $A$--bundle. This implies that $X$ is a holomorphic principal compact torus bundle over $B$. This leads to a 
contradiction as in the proof of item $(ii)$. This completes the proof of the theorem.
\end{proof}

As a by-product of the proof of Theorem~\ref{main lemma} we obtain another proof of Corollary~\ref{thmFujiki}.

\begin{proof}[{Another proof of Corollary~\ref{thmFujiki}}]
Take a compact, simply connected complex manifold $X$ in Fujiki class $\mathcal C$ with trivial canonical class and
admitting a holomorphic rigid geometric structure. Theorem~\ref{main lemma} shows that a maximal connected abelian subgroup $A$ of
$\rm{Aut}(X,\, \phi)$ has closed orbits in 
$X$, and these orbits coincide with the fibers of the fibration $\pi$ defined in the proof of item $(i)$.
Theorem 2.3 in \cite{GW} (see also \cite{Ho} 
for the particular case of $\mathbb C$--actions) proves that there is a compact {\it complex} torus $K$ in $A$ such that the 
$A$--orbits coincide with the $K$--orbits. Moreover, the induced $K$--action on $X$ is free and the quotient map $X\, \longrightarrow\, 
X/K$ gives a holomorphic principal $K$--bundle. This implies that the fibration $\pi$ is locally trivial, so it
is a holomorphic principal bundle. This is in contradiction with item $(ii)$ in Theorem~\ref{main lemma}.
\end{proof}

The following statement is an easy application of Theorem~\ref{main lemma}.

\begin{corollary}\label{thm part cases}
Let $X$ be a compact complex manifold with trivial canonical bundle bearing a rigid holomorphic geometric structure. Assume that one of the following holds:
\begin{enumerate}[label= $\circ$]
\item The dimension of $X$ is at most three, or
\item The algebraic dimension of $X$ is at most one.
\end{enumerate}
Then the fundamental group of $X$ is infinite.
\end{corollary}

\begin{proof}
To prove by contradiction, assume that the fundamental group of $X$ is finite. Replacing $X$ by it
universal cover we will assume that $X$ is simply connected.

{\it Case $a(X) \le 1$. }

\noindent
Denote by $A$ the maximal abelian subgroup of the automorphism group ${\rm Aut}(X,\, \phi)$ in the statement of Theorem \ref{act alg 
red}. Then Theorem \ref{act alg red} shows that $A$ acts on $X$ with orbits containing the fibers of the algebraic reduction of $X$. 
Since the algebraic dimension of $X$ is at most one, the dimension of the $A$--orbits is at least $\dim X-1$. If the action of $A$ on 
$X$ has an open orbit, the geometric structure $\phi$ is locally homogeneous and Proposition \ref{loc hom2} furnishes a contradiction.

Let us now consider the case where the dimension of the $A$-orbits in $X$ is $\dim X -1$. Theorem~\ref{main lemma} constructs a 
holomorphic submersion $\pi \,: \,X \,\longrightarrow\, B$ over a compact simply connected complex manifold $B$ with globally
generated canonical 
bundle $K_B$ such that the fibers of $\pi$ coincide with the $A$--orbits. Since the $A$-orbits in $X$ coincide with the fibers of $\pi$,
we have $\dim B \,=\, 1$. Therefore, $B$ is a compact Riemann surface. This is a contradiction because the canonical bundle of
the only simply connected compact Riemann surface ${\mathbb C}{\mathbb P}^1$ is not globally generated.\\

{\it Case $\dim X \le 3$. }

\noindent
Give the previously treated case and Theorem~\ref{moishezon}, we only need to address the case where the threefold $X$ has algebraic dimension two.

Then the $A$--orbits have dimension at least one. If the orbits have dimension two or three, we conclude as in the proof of the 
previous case. If the dimension of the $A$--orbits in $X$ is one, then the fibration $\pi \,:\, X \,\longrightarrow \,B$ constructed in 
Theorem~\ref{main lemma} is an elliptic fibration; its fibers are elliptic curves. Since the base $B$ is simply connected the map from 
$B$ to the moduli space of elliptic curves lifts to a holomorphic map from $B$ to the upper-half plane (the Teichm\"uller space of 
elliptic curves). Since $B$ is compact, it is a constant map. This implies that $\pi$ is locally trivial, which contradicts $(ii)$ in 
Theorem~\ref{main lemma}. This finishes the proof.
\end{proof}

The following proposition is a consequence of Corollary~\ref{thm part cases}.

\begin{proposition}\label{proj con}
Let $\phi$ be a holomorphic projective connection on a compact complex manifold $X$
with trivial canonical bundle. Then the following two assertions hold:
\begin{enumerate}[label=$(\roman*)$]
\item $X$ admits a holomorphic affine connection $\nabla$ which is 
projectively isomorphic to $\phi$.

\item If $X$ has algebraic dimension at most one, or $X$ is a threefold, then the fundamental group of $X$ is infinite.
\end{enumerate}
\end{proposition}

\begin{proof} $(i)$. Since $K_X\,=\, {\mathcal O}_X$, there is a (global)
holomorphic torsionfree affine connection projectively equivalent to
the projective connection $\phi$ \cite[p.~7449, Lemma 5.6]{BD4}. 

This can also be seen as a direct consequence of the results in \cite{Gu} and \cite{KO},
which is explained below.
There exists a holomorphic affine connection representing the projective connection 
$\phi$ if and only if the cocycle (3.2) in \cite{KO} (defined as $d{\rm 
log}(\Delta_{ij})$, where $\{\Delta_{ij}\}$ is the 1-cocycle of the canonical bundle $K_X$) 
vanishes in ${\rm H}^1(X, \,\Omega_X)$; 
see the explicit formula (3.6) in \cite[p.~78--79]{KO}. This condition is satisfied if 
and only if the canonical bundle $K_X$ admits a holomorphic affine connection (this 
coincides with the vanishing condition of the Atiyah class for $K_X$ \cite[p.~195, 
Theorem 5]{At}; see also \cite[p.~96--97]{Gu} for an alternative approach). In the 
particular case where $X$ is compact and K\"ahler, this is equivalent to the vanishing 
of the real first Chern class of $X$. In any case, the above condition is automatically 
satisfied when, as it happens in our situation, $K_X$ is trivial.

$(ii)$.\, Since a holomorphic affine structure is rigid geometric of affine type, this is 
indeed a particular case of Corollary~\ref{thm part cases}.
\end{proof}

\begin{remark}\label{end conj}
The last argument in the proof of Corollary~\ref{thm part cases} (involving the Teichm\"uller space) generalizes to the case where the 
fibers of the fibration $\pi$ are polarized abelian varieties. Hence non-isotrivial fibrations $\pi$ as in Theorem~\ref{main lemma} do 
not exist if $X$ is a projective manifold with trivial canonical bundle. The fibrations constructed in Theorem~\ref{main lemma} neither 
exist when $X$ is a compact K\"ahler Calabi--Yau manifold (see Theorem 3.1 in \cite{TZ} or Section 4 in \cite{Bern09}). We conjecture 
that they do not exist if $X$ is a compact simply connected complex manifold with trivial canonical bundle. This conjecture implies 
that compact simply connected complex manifolds with trivial canonical bundle do not admit holomorphic rigid geometric structures.
\end{remark}

\section*{Acknowledgements}

The authors are grateful to Fr\'ed\'eric Campana for very helpful comments and, in particular, for his suggestion to use Theorem 4.9 of 
\cite{CP} in the proof of Proposition~\ref{prepop}. The authors thank also Uri Bader, St\'ephane Druel, Hsueh-Yung Lin, Laurent Meersseman, Mihnea Popa and Valentino Tosatti for very helpful discussions. We thank the referee for his meticulous reading of the manuscript and for his many suggestions to improve the presentation. 

The first and third named authors thank ICTS Bangalore and IISc Bangalore for hospitality.
The third-named author wishes to thank T.I.F.R. Mumbai for hospitality. 
This work has been supported by the French government through the UCAJEDI Investments in the 
Future project managed by the National Research Agency (ANR) with the reference number 
ANR2152IDEX201. The first-named author is partially supported by a J. C. Bose
Fellowship (JBR/2023/000003). 
The second and fourth-named authors acknowledge the support of the French Agence Nationale de la Recherche (ANR) under reference ANR-21-CE40-0010.
The second-named author wishes to thank the Institut Universitaire de France for providing excellent working condition.

\section*{Declarations}

Declaration of interests. The authors do not work for, advise, own shares in, or receive
funds from any organisation that could benefit from this article, and have declared no
affiliation other than their research organisations.

Data availability statement. No data were generated or used.

\end{document}